\documentclass[reqno]{amsart}

\usepackage[utf8]{inputenc} 
\usepackage[T1]{fontenc}    
\usepackage{url}            
\usepackage{booktabs}       
\usepackage{amsfonts}       
\usepackage{nicefrac}       
\usepackage{microtype}      
\usepackage[]{xcolor}         

\usepackage{amsmath,amssymb,amsfonts,mathrsfs}
\usepackage{nicefrac}
\usepackage{algorithmic}
\usepackage{algorithm}
\usepackage{graphicx}

\usepackage{multirow}

\usepackage{wrapfig}

\usepackage{booktabs}

\usepackage{tikz}
\usetikzlibrary{positioning, arrows.meta, calc, matrix}

\usepackage{xcolor}

\definecolor{strblue}{HTML}{0F1ED2}
\definecolor{strred}{HTML}{E61E8C}

\usepackage{natbib}

\usepackage{comment}

\usepackage{aliascnt}
\usepackage{hyperref}
\usepackage{cleveref}

\theoremstyle{definition}

\newtheorem{theorem}{Theorem}[section]
\crefname{theorem}{Theorem}{Theorems}
\Crefname{theorem}{Theorem}{Theorems}

\newaliascnt{definition}{theorem}
\newtheorem{definition}[definition]{Definition}
\aliascntresetthe{definition}
\crefname{definition}{Definition}{Definitions}
\Crefname{definition}{Definition}{Definitions}

\newaliascnt{proposition}{theorem}
\newtheorem{proposition}[proposition]{Proposition}
\aliascntresetthe{proposition}
\crefname{proposition}{Proposition}{Propositions}
\Crefname{proposition}{Proposition}{Propositions}

\newaliascnt{lemma}{theorem}
\newtheorem{lemma}[lemma]{Lemma}
\aliascntresetthe{lemma}
\crefname{lemma}{Lemma}{Lemmas}
\Crefname{lemma}{Lemma}{Lemmas}

\newaliascnt{corollary}{theorem}
\newtheorem{corollary}[corollary]{Corollary}
\aliascntresetthe{corollary}
\crefname{corollary}{Corollary}{Corollaries}
\Crefname{corollary}{Corollary}{Corollaries}

\newaliascnt{remark}{theorem}
\newtheorem{remark}[remark]{Remark}
\aliascntresetthe{remark}
\crefname{remark}{Remark}{Remarks}
\Crefname{remark}{Remark}{Remarks}

\newaliascnt{example}{theorem}
\newtheorem{example}[example]{Example}
\aliascntresetthe{example}
\crefname{example}{Example}{Examples}
\Crefname{example}{Example}{Examples}

\newaliascnt{assumption}{theorem}
\newtheorem{assumption}[assumption]{Assumption}
\aliascntresetthe{assumption}
\crefname{assumption}{Assumption}{Assumptions}
\Crefname{assumption}{Assumption}{Assumptions}

\numberwithin{equation}{section}

\crefname{equation}{Equation}{Equations}
\crefname{figure}{Figure}{Figures}
\crefname{table}{Table}{Tables}
\crefname{algorithm}{Algorithm}{Algorithms}

\crefname{section}{}{}
\creflabelformat{section}{\S #2#1#3}
\crefname{subsection}{}{}
\creflabelformat{subsection}{\S #2#1#3}
\crefname{appendix}{Appendix}{Appendixes}

\DeclareMathOperator*{\argmin}{arg\,min} 
\DeclareMathOperator*{\argmax}{arg\,max} 
\DeclareMathOperator{\Conv}{Conv}

\DeclareMathOperator{\Proj}{Proj}

\title[Explicit Iteration Complexity of Exact inverse optimization for ILPs]{Explicit Iteration Complexity of Exact Data-Driven Inverse Optimization for Integer Linear Programs}

\author{%
  Akira Kitaoka	
}
\address{NEC Corporation, 1753 Shimonumabe, Nakahara-ku, Kawasaki, Kanagawa, Japan }
\email{akira-kitaoka@nec.com}


\begin{document}

\begin{abstract}

A data-driven inverse optimization problem (DDIOP) is the problem of estimating the objective-function parameters (weights) that explain observed optimal-solution data, and it arises in many applications, including integer linear programming (ILP). It is known that, by applying gradient-based optimization methods to the suboptimality loss, the inverse optimization of ILPs can be solved exactly within finitely many oracle iterations, and that the required number of iterations is bounded as $T=O(1/\gamma(\ell_{\mathrm{sub}})^2)$ in terms of a problem-dependent geometric constant $\gamma(\ell_{\mathrm{sub}})$. However, no means of bounding $\gamma(\ell_{\mathrm{sub}})$ from below as a function of the problem size has been available, and hence the number of iterations could not be given as an explicit function of the problem size. We therefore give, when the forward problem is an integer linear program (ILP), the number of iterations sufficient for projected subgradient descent applied to the suboptimality loss to achieve exact consistency with the observed data, as a fully explicit function of the number of samples, the dimension of the features, the ranges of the features, and the structure of the constraint coefficient matrix, up to polynomial factors in the basic constants (the diameter of the weight set, the step-size parameter, and the Lipschitz constant of the suboptimality loss).

\end{abstract}

\maketitle


\section{Introduction}
Inverse optimization is the problem of estimating an objective function or its parameters from observed optimal solutions~\citep{ahuja2001inverse,heuberger2004inverse,chan2019inverse,chan2023inverse}, with a wide range of applications including transportation~\citep{bertsimas2015data}, power systems~\citep{birge2017inverse}, healthcare~\citep{chan2022inverse}, advertisement scheduling~\citep{Suzuki-2019-TV}, and inverse reinforcement learning~\citep{ng2000algorithms}.

In this paper, we treat the case where the objective function of the forward problem is given as a linear combination of known features: for each state $s\in\mathcal{S}$, given a feasible region $X(s)\subset\mathcal{X}$, a feature map $f=(f_1,\ldots,f_d)\colon\mathcal{X}\times\mathcal{S}\to\mathbb{R}^d$, and weights $\theta\in\Theta\subset\mathbb{R}^d$, the forward problem is
$x^*(\theta,s)\in \argmax_{x\in X(s)} \theta^\top f(x, s)$.
Given observed data $\mathcal{D}=\{(s^{(n)},x^{(n)})\}_{n=1}^N$, $x^{(n)}=x^*(\theta^*,s^{(n)})$, generated by unknown true weights $\theta^*$, the data-driven inverse optimization problem (DDIOP) asks for weights $\theta$ that explain every observed solution as an optimal solution (see \Cref{sec:problem_setup_and_algorithm} for the precise formulation).
In particular, setting $f(x, s)=x$, this framework includes mixed integer linear programs (MILPs).

To measure consistency with the observed data, we compare the observed features $a^{(n)}:=f(x^{(n)}, s^{(n)})$ with the predicted features $a^*(\theta,s^{(n)}):=f(x^*(\theta,s^{(n)}), s^{(n)})$.
A representative loss, the prediction loss of features (PLF)
$\ell_{\mathrm{plf}}(\theta):=\frac{1}{N}\sum_{n=1}^N \lVert a^*(\theta,s^{(n)})-a^{(n)} \rVert_2^2$
\citep{aswani2018inverse,chan2019inverse,babier2021ensemble,chan2023inverse,ferber2023surco,liang2023data},
is in general discontinuous in inverse optimization for MILPs because $a^*(\theta,s^{(n)})$ can vary discontinuously with respect to $\theta$, which makes direct gradient-based minimization difficult~\citep[cf.][]{Beck-2017-First,hazan2022introduction,garrigos2023handbook}.
In contrast, the suboptimality loss
$\ell_{\mathrm{sub}}(\theta):=\frac{1}{N}\sum_{n=1}^N ( \theta^\top a^*(\theta,s^{(n)}) - \theta^\top a^{(n)} )$,
introduced by \citet{Mohajerin-2018-Data}, is convex and Lipschitz continuous in our setting
\citep{Mohajerin-2018-Data,Barmann-2018-online,Kitaoka-2023-convergence-IRL},
and $\ell_{\mathrm{sub}}(\theta)=0$ is equivalent to every observed solution $x^{(n)}$ being optimal under the weights $\theta$, that is, to exact consistency with the observed data.
Consequently, first-order and online optimization methods---projected subgradient descent, multiplicative weights update, online Newton step, MetaGrad, and others---have been applied to the suboptimality loss, and regret and best-iterate analyses provide asymptotic guarantees that the loss decreases as the number of iterations $T$ grows
\citep{Barmann-2018-online,besbes2021online,besbes2025contextual,gollapudi2021contextual,sakaue2025online}.

Moreover, \citet{kitaoka2024exact} strengthened these asymptotic guarantees, using the geometric structure that the suboptimality loss is convex and piecewise linear and that its minimizer set has a relative interior point: gradient-based optimization methods reach $\ell_{\mathrm{sub}}=0$---that is, exact consistency with the observed data---within finitely many iterations, and the number of iterations is bounded from above via a problem-dependent geometric constant $\gamma(\ell_{\mathrm{sub}})>0$ (e.g., $T=O(1/\gamma(\ell_{\mathrm{sub}})^2)$ for projected subgradient descent).
However, $\gamma(\ell_{\mathrm{sub}})$ is an abstract constant determined by the geometry of the problem, and no means of bounding it from below as a function of the problem size was known. Consequently, the number of iterations sufficient for exact consistency could not be given as an explicit function of the problem size, such as the sample size $N$ and the feature dimension $d$.%

In this paper, when the forward problem is an integer linear program (ILP), we derive explicit lower bounds on $\gamma(\ell_{\mathrm{sub}})$ from the structure of integer programming (total unimodularity, Graver bases, and M-convexity/M$^\natural$-convexity).
Substituting them into the iteration upper bounds of \citet{kitaoka2024exact}, we give the number of iterations sufficient to achieve exact consistency with the observed data as a fully explicit function of the sample size $N$, the feature dimension $d$, the ranges of the features, and the structure of the constraint coefficient matrix, up to polynomial factors in the basic constants (the diameter of the weight set $\mathrm{diam}(\Theta)$, the step-size parameter $\beta$, and the Lipschitz constant of the suboptimality loss $L(\ell_{\mathrm{sub}})$).
The obtained lower bounds are polynomial in the dimension and independent of the ranges of the features for M-convex and M$^\natural$-convex sets, whereas for general ILPs they can become exponentially small in the dimension $d$. The latter is consistent with the NP-hardness of inverse optimization~\citep{aswani2018inverse}, and shows that, in our bounds, the exponential dependence is confined to the dimension $d$ (\Cref{rem:gamma_hardness_main}).%

\subsection*{Contributions}
The contributions of this paper are as follows.

\begin{itemize}
\item For data-driven inverse optimization whose forward problem is an integer linear program (ILP), under $\Theta=\Delta^{d-1}$ (\Cref{theo:gamma_lowerbound_on_simplex}), we give an explicit lower bound $\gamma(\ell_{\mathrm{sub}})\ge \frac{1}{N^{d}\max(d-1,\sqrt2)\,\lVert m\rVert_2^{d-1}}$ on the constant $\gamma(\ell_{\mathrm{sub}})$, where $m=(m_i)_i$ collects the ranges of the features and $N$ is the sample size.

\item When the constraints form a linear inequality system $A x\le b$, we give a lower bound $\gamma(\ell_{\mathrm{sub}})=\Omega\big(1/(N\,d^{(d+1)/2}(2C)^{d-1})\big)$ that is independent of the ranges of the features and determined solely by the $\ell_\infty$ norm $C=g_\infty(\widetilde A)$ of the Graver basis of the slack-augmented matrix $\widetilde A=[A\mid I]$ (in particular $C=1$ for totally unimodular matrices).

\item When the feasible region is an M-convex or M$^\natural$-convex set, we give a polynomial-in-dimension lower bound $\Omega\big(1/(N d^{2})\big)$ that is independent of the ranges of the features.

\item Substituting these lower bounds into the iteration upper bounds for gradient-based optimization methods~\citep{kitaoka2024exact}, we give the number of iterations sufficient to solve the data-driven inverse optimization of ILPs exactly as an explicit function of the problem size. The same substitution also yields an explicit bound on the number of iterations required to attain the minimum of the prediction loss of features (PLF).

\item We show that the general-ILP lower bounds above are essentially tight: we construct an explicit family of ILP instances on which $\gamma(\ell_{\mathrm{sub}})=\|m\|_\infty^{-\Omega(d)}$ (with $\|m\|_\infty=\max_i m_i$) is exponentially small in the dimension $d$ (\Cref{prop:tightness}). Hence the exponential dependence of the resulting iteration bound on $d$ is genuine rather than an artifact of a loose lower bound.
\end{itemize}

The remainder of this paper is organized as follows.
In \cref{sec:related_work}, we review related work.
In \cref{sec:problem_setup_and_algorithm}, we describe the DDIOP setting and the algorithm, which applies gradient-based optimization to the suboptimality loss.
In \cref{sec:main_results}, we state the main results.
In \cref{sec:comparison_tables}, we compare our bounds with existing methods; in \cref{sec:gamma_lowerbound_ILP_appx}, we give the lower bounds on the constant $\gamma(\ell_{\mathrm{sub}})$ together with their proofs; and in \cref{sec:iteration_bound_from_gamma}, we derive the iteration upper bounds.

\section{Related Work}\label{sec:related_work}

Research on inverse optimization has classically developed for combinatorial and network optimization~\citep{ahuja2001inverse,heuberger2004inverse}, and more recently the data-driven framework---which estimates an objective function or its parameters from observed data---has been actively studied from both theoretical and applied perspectives~\citep{chan2019inverse,chan2023inverse} (the applications and loss functions were reviewed in the introduction; for systematic surveys see \citet{chan2019inverse,chan2023inverse}, and for a detailed account of related work in the same setting as ours see \citet{kitaoka2024exact}). In this section, we focus on the work directly related to the contributions of this paper: explicit lower bounds on $\gamma(\ell_{\mathrm{sub}})$ and explicit evaluations of the number of iterations.

\paragraph{Existing guarantees for the suboptimality loss.}
For the convex suboptimality loss, regret and best-iterate analyses based on first-order and online optimization methods---projected subgradient descent, multiplicative weights update, online Newton step, and MetaGrad---have been carried out~\citep{Barmann-2018-online,besbes2021online,besbes2025contextual,gollapudi2021contextual,Kitaoka-2023-convergence-IRL,sakaue2025online} (for details of existing regret analyses see \Cref{sec:known_regret_analysis}, and for the offline implications see \Cref{tab:pro_con_SL_detail}); however, these guarantees remain asymptotic, of the form $\ell_{\mathrm{sub}}(\theta^t)\to 0$ or $\min_{1\leq t\leq T}\ell_{\mathrm{sub}}(\theta^t)\to 0$. \citet{kitaoka2024exact} strengthened these asymptotic guarantees into finite-step exact attainment, with an iteration upper bound via the constant $\gamma(\ell_{\mathrm{sub}})$, under the structure that the suboptimality loss is convex and piecewise linear and its minimizer set has a relative interior point.

\paragraph{Inverse optimization with discrete structure.}
The work most closely sharing our motivation is~\citet{oki2026finite}, which treats online inverse linear optimization where the feasible region is an \emph{M-convex set} (a fundamental class in discrete convex analysis that includes matroids~\citep[cf.][]{murota1996convexity,murota1998discrete,murota2003discrete}) and gives a regret bound $O(d\log d)$ independent of the number of iterations $T$; it lies on the same complexity axis as ours in that it obtains a polynomial-in-dimension complexity by exploiting discrete structure. This paper extends this axis: by deriving lower bounds on $\gamma(\ell_{\mathrm{sub}})$ from the structure of integer programming, such as total unimodularity and Graver bases~\citep{sturmfels1996grobner,onn2010nonlinear}, we broaden the scope to general ILPs, which contain M-convex and M$^\natural$-convex sets as special cases, and we guarantee finite-step attainment of exact consistency with the observed data ($\ell_{\mathrm{sub}}=0$), together with explicit iteration upper bounds, rather than the finiteness of the cumulative regret.
To avoid the discontinuity of the PLF or of the discrete optimal-solution map, approaches that smooth the problem by adding noise~\citep{berthet2020learning} or regularization~\citep{wilder2019melding}, or by using cutting planes and LP relaxations for MILPs~\citep{ferber2020mipaal,ferber2023surco}, have also been proposed; however, these generally solve an approximate problem and do not guarantee exact consistency for the original discrete problem in finitely many iterations (for an examination of the approximation errors and solution uniqueness of these methods, see \citealp{kitaoka2024exact}). Our framework does not use smoothing and is based on the geometric structure of the suboptimality loss itself.

\paragraph{Computational complexity and evaluation of the number of iterations.}
This paper sharpens the qualitative guarantee of attaining exact consistency in finitely many iterations into an explicit evaluation that quantifies the required number of iterations in accordance with the problem structure. Our iteration upper bound is described via the problem-dependent geometric constant $\gamma(\ell_{\mathrm{sub}})$, and for PSGD (SRSS, SRSL; defined in \Cref{exa:PSGD}) it takes the form $T=O(1/\gamma(\ell_{\mathrm{sub}})^2)$ up to polynomial factors. Whereas existing first-order/online studies rarely give the number of iterations required for exact consistency as an explicit function of the problem size, this paper, when the forward problem is an integer linear program (ILP), derives an explicit lower bound on $\gamma(\ell_{\mathrm{sub}})$ from the structure of integer programming (total unimodularity, the Graver basis~\citep{sturmfels1996grobner,onn2010nonlinear}, and M-convexity / M$^\natural$-convexity), and thereby evaluates the required number of iterations as an explicit function of the problem size. As a result, it is shown quantitatively that for a broad range of cases in which the structure can be exploited, or for a fixed dimension $d$ (including the totally unimodular case), the number of iterations is bounded by a polynomial in the problem size. In particular, for M-convex and M$^\natural$-convex sets, the lower bound on $\gamma(\ell_{\mathrm{sub}})$ is a polynomial in the dimension, $\Omega(1/(Nd^{2}))$, independent of the ranges of the features, and the number of iterations is bounded by $O(N^{2}d^{4})$ times $\mathrm{poly}(\mathrm{diam}(\Theta),\beta,L(\ell_{\mathrm{sub}}))$ (\Cref{sec:iteration_bound_from_gamma}). We also note that $\gamma(\ell_{\mathrm{sub}})$ can be upper-bounded via the SPO loss~\citep{Mohajerin-2018-Data,elmachtoub2022smart} (\citealp[Section~13.4]{kitaoka2024exact}); however, what is needed to evaluate the number of iterations is a lower bound on $\gamma(\ell_{\mathrm{sub}})$, which this paper provides, and the two are complementary.

Moreover, our explicit evaluation makes the possible worst-case behavior of the resulting bounds visible. In the worst case for general ILPs without structural assumptions, or for linear inequality constraints including the totally unimodular case, the lower bound on $\gamma(\ell_{\mathrm{sub}})$ that we give can become exponentially small in the dimension $d$, so the corresponding iteration upper bounds can grow exponentially in $d$. This dependence is consistent with the computational hardness of inverse optimization---\citet[Theorem 1]{aswani2018inverse} show NP-hardness of inverse optimization with noisy data. In this way, rather than claiming polynomial-time solvability in the general sense, this paper quantitatively separates the two regimes: it guarantees polynomial iteration bounds when the structure can be exploited, while in the general case our bounds leave room for exponential growth in the dimension $d$ (\Cref{rem:gamma_hardness_main}).

\paragraph{\texorpdfstring{Implications enabled by the explicit lower bounds on $\gamma(\ell_{\mathrm{sub}})$.}{Implications enabled by the explicit lower bounds.}}
Previously, $\gamma(\ell_{\mathrm{sub}})$ was an abstract constant guaranteed only to be positive (\Cref{prop:grad_based_opt_achieve_min_SL_before}), so the iteration upper bounds of \citet{kitaoka2024exact} remained qualitative statements of finite termination. Having lower bounds on $\gamma(\ell_{\mathrm{sub}})$ computable from the problem size newly enables the following considerations.
First, an iteration budget can be determined before running the algorithm. For instance, when the coefficient matrix is totally unimodular, the coefficient $C=g_\infty(\widetilde{A})$ determined by the linear-inequality constraint matrix (the $\ell_\infty$ norm of the Graver basis of the slack-augmented matrix $\widetilde{A}=[A\mid I]$) equals $1$, so under the assumptions of \Cref{sec:iteration_bound_from_gamma} the upper bounds in \Cref{tab:iteration_upperbound} are determined solely by the sample size $N$, the dimension $d$, and the basic constants $\mathrm{diam}(\Theta)$, $\beta$, and $L(\ell_{\mathrm{sub}})$; they thus serve as an a priori stopping criterion---running this many iterations guarantees exact consistency---and as an estimate of the computational budget.
Second, comparisons with asymptotic guarantees become comparisons between explicit functions of the problem size. The finite-step guarantee ``$0$ if $T\geq\cdots$'' in \Cref{tab:pro_con_SL_detail} could not previously be compared with regret-type guarantees on the same footing because $\gamma(\ell_{\mathrm{sub}})$ was unknown; substituting the lower bounds places both on the same problem-size scale.
Third, the computational value of discrete structure can be quantified. An M$^\natural$-convex set can also be described by a system of linear inequalities, but the lower bound $\Omega(1/(Nd^{2}))$ obtained directly from M$^\natural$-convexity is far better than the bound $\Omega\big(1/(N\,d^{(d+1)/2}(2C)^{d-1})\big)$ obtained by treating the same set via the general theory for linear inequality constraints (\Cref{rem:linear_inequality_includes_Mnat}). That is, for one and the same feasible region, which structural information is used to evaluate $\gamma(\ell_{\mathrm{sub}})$ can make the resulting iteration estimate polynomial or exponential in the dimension, and this distinction becomes visible as a comparison of the lower bounds.

\paragraph{Relation to weak sharp minima.}
\label{sec:weak_sharp_minima}

The structure that is key to the finite-step attainment (\Cref{theo:grad_based_opt_achieve_min_SL})---that the set of minimizers has a relative interior point---is closely related to the notion of weak sharp minima~\citep{burke1993weak}, and finite termination of optimization algorithms under weak sharp minima has been studied classically~\citep{polyak1987introduction,ferris1991finite,burke1993weak}. For a detailed comparison between these finite-termination theories and the finite-step attainment of gradient-based methods for the nonsmooth suboptimality loss, see \citet{kitaoka2024exact}. The difference of the present paper lies not in whether finite termination occurs but in its quantification: whereas the classical finite-termination theory of weak sharp minima does not provide the number of iterations until termination as a function of the problem size, this paper derives explicit lower bounds on $\gamma(\ell_{\mathrm{sub}})$ from the structure of integer programming (Graver bases, total unimodularity, and M-convexity/M$^\natural$-convexity; \Cref{sec:gamma_lowerbound_ILP_appx,sec:iteration_bound_from_gamma}), thereby giving the required number of iterations as an explicit function of the problem size.

\paragraph{Positioning of this paper.}
In light of the above, this paper is positioned at the intersection of two streams of work: the asymptotic analysis of first-order/online optimization for the suboptimality loss, and the problem awareness of achieving exact consistency with the observed data in discrete inverse optimization including ILPs. The finite-time exact solvability of gradient-based methods for the suboptimality loss, together with the general iteration upper bound $T=O(1/\gamma(\ell_{\mathrm{sub}})^2)$ via the constant $\gamma(\ell_{\mathrm{sub}})$, was established by \citet{kitaoka2024exact}; this paper is responsible for bounding $\gamma(\ell_{\mathrm{sub}})$ from below in terms of the structure of ILPs. Its novelty is to derive explicit lower bounds on the constant $\gamma(\ell_{\mathrm{sub}})$ from the structure of integer programming (total unimodularity, Graver bases, and M-convexity/M$^\natural$-convexity), and thereby to give the number of iterations required for exact consistency as an explicit function of the problem size.

\section{Problem Setup and Algorithms}
\label{sec:problem_setup_and_algorithm}

For each state $s \in \mathcal{S}$, we define the optimal solution to the forward problem and its associated feature vector as follows:
\begin{equation}\label{eq:FOP_linear}
    x^* (\theta , s ) \in \argmax_{x \in X(s)} \theta^{\top} f ( x, s ) = \sum_{i=1}^d \theta_i f_i (x, s) ,
    \quad a^* (\theta , s ) = f (x^* (\theta , s ), s).
\end{equation}
Given the observed data $\mathcal{D} = \{ (s^{(n)}, x^{(n)}) \}_{n=1}^N$, the data-driven inverse optimization problem associated with the forward problem~\cref{eq:FOP_linear} is the problem of finding $\theta \in \Theta$ satisfying the following condition for all $n = 1, \ldots, N$:
\begin{equation}
    x^{(n)} \in \argmax_{x \in X(s^{(n)})} \theta^{\top} f ( x, s^{(n)} ) = \sum_{i=1}^d \theta_i f_i (x, s^{(n)}).
    \label{eq:IOP_linear}
\end{equation}

We impose the following assumptions to specify the problem setting studied in this paper.
\begin{assumption}
    \label{assu:WIRL}
    Let the nonempty set $\mathcal{S}$ be the state space, and let the nonempty set $\mathcal{X} \subset \mathbb{R}^{d_{\mathcal{X}}}$ be the decision space. Let $f = (f_1 , \ldots , f_d)\colon \mathcal{X}\times\mathcal{S} \to \mathbb{R}^d$ be a mapping such that each component $f_i$ is a piecewise linear function.
    Let the weight space $\Theta \subset \mathbb{R}^d$ be a bounded, closed, and convex set.
    For each state $s \in \mathcal{S}$, let the feasible region $X(s)$ be a finite union of bounded, closed, and convex polytopes that are subsets of $\mathcal{X}$.
    Let $\mathcal{D} := \left\{ \left( s^{(n)}, x^{(n)} \right) \right\}_{n=1}^N$ be a set of training data such that the samples $s^{(n)} \in \mathcal{S}$ are generated from an unknown distribution $\mathbb{P}_{\mathcal{S}}$, and there exists an unknown $\theta^* \in \Theta$ satisfying, for each $n =1 , \ldots , N$,
    $
       x^{(n)} = x^* ( \theta^* , s^{(n)})
    $.
\end{assumption}

As a technical condition, we further assume the following.
\begin{assumption}
    \label{assu:WIRL-uniqueness}
    For each $n=1,\ldots,N$, the feature $a^*(\theta^*, s^{(n)})$ is uniquely determined.
\end{assumption}

\cref{assu:WIRL-uniqueness} is a natural assumption, in view of the following statement.
\begin{lemma}[{\citealp[Lemma~3.3]{kitaoka2024exact}}]
    \label{lem:Psi_set_is_almost_Phi}
    Suppose \cref{assu:WIRL} holds and let $\Theta=\Delta^{d-1}\subset \mathbb{R}^d$.
    Then, for $\theta\in\Delta^{d-1}$ almost everywhere (with respect to the measure on $\Delta^{d-1}$ induced by the Lebesgue measure), the feature $a^*(\theta,s^{(n)})$ is uniquely determined for every $n=1,\ldots,N$.
\end{lemma}

In this paper, based on the following proposition, we apply a gradient-based optimization method to the suboptimality loss $\ell_{\mathrm{sub}}$ and minimize it.

\begin{proposition}[{\citealt[Proposition~3.1]{Barmann-2018-online}; \citealt[Lemma~4.8]{Kitaoka-2023-convergence-IRL}}]
    \label{prop:SL_is_Lipschitz_convex}
    Suppose that \cref{assu:WIRL} holds.
    Then the following statements hold:
    (A) the suboptimality loss $\ell_{\mathrm{sub}}$ is convex;
    (B) the suboptimality loss $\ell_{\mathrm{sub}}$ is Lipschitz continuous; and
    (C) a subgradient of the suboptimality loss $\ell_{\mathrm{sub}}$ at $\theta\in\Theta$ is given by
    $ g(\theta)
        :=\sum_{n=1}^N \bigl(a^*(\theta, s^{(n)}) - a^{(n)}\bigr) / N $.    
\end{proposition}

\begin{wrapfigure}{r}{8cm}
    \vspace{-2\intextsep}
    \begin{minipage}{\linewidth}
    \begin{algorithm}[H]
        \caption{Minimization of the suboptimality loss \citep{kitaoka2024exact}}\label{alg:intention-WIRL-gradual-decay}
        \begin{algorithmic}[1]
            \STATE Initialize $\theta^1 \in \Theta$
            \FOR{$t = 1 , \ldots, T-1$}
                \STATE For each $n = 1 , \ldots , N$, solve for $x^*(\theta^t, s^{(n)})$
                \STATE $\theta^{t+1} \leftarrow \mathrm{update}_{t}\!\left(\{\theta^{t^\prime}\}_{t^\prime = 1}^t \mid \ell_{\mathrm{sub}}, g\right)$
            \ENDFOR
            \RETURN $\displaystyle \theta_{\mathrm{best}}^{T} \in \argmin_{\theta \in \{\theta^{t}\}_{t=1}^T} \ell_{\mathrm{sub}}(\theta)$
        \end{algorithmic}
    \end{algorithm}
    \end{minipage}
    \vspace{-\intextsep}
\end{wrapfigure}
Hereafter, when it is clear from the context,
$\min_{\Theta} \ell = \min_{\theta \in \Theta} \ell (\theta)$ and
$\argmin_{\Theta} \ell = \argmin_{\theta \in \Theta} \ell (\theta)$
are used as abbreviations.
Let $\ell \colon \Theta \to \mathbb{R}$ be a Lipschitz-continuous convex function, and let $\partial \ell$ denote the subdifferential of $\ell$.
Let $\nabla \ell \colon \Theta \to \mathbb{R}^d$ be any (sub)gradient selection satisfying, for every $\theta$, $\nabla \ell (\theta) \in \partial \ell (\theta)$.
For $t \in \mathbb{Z}_{\geq 1}$, let the update map be
$\mathrm{update}_{t} \colon \Theta^t \times \mathbb{R}^{t} \times (\mathbb{R}^{d})^{t} \to \Theta$.
Moreover, define
\[
    \begin{split}
        & \mathrm{update}_t \left( \{ \theta^{t^\prime} \}_{t^\prime =1}^t \mid \ell , \nabla \ell \right) 
         :=\mathrm{update}_{t} ( \{ \theta^{t^\prime} \}_{t^\prime = 1}^t , \{ \ell (\theta^{t^\prime}) \}_{t^\prime = 1}^t , \{ \nabla \ell (\theta^{t^\prime}) \}_{t^\prime = 1}^t )    
    \end{split}
\]
The resulting procedure is described in \cref{alg:intention-WIRL-gradual-decay}.

\begin{example}[Projected subgradient descent]
    \label{exa:PSGD}
    Let $\alpha_{t} \colon \mathbb{R}^d \times \mathbb{R} \times \mathbb{R}^{d} \to \mathbb{R}_{\geq 0}$ be a mapping, and refer to $\{ \alpha_{t} \}_{t \in \mathbb{Z}_{\geq 1}}$ as the \emph{learning rate}.
    Let $\Proj_{\Theta}$ denote the projection onto $\Theta$.
    We identify the mapping $\alpha_t$ with its value $\alpha_t ( \theta^t , \ell(\theta^t) , \nabla \ell(\theta^t))$ when no confusion arises.
    We define projected subgradient descent (PSGD) \citep[cf.][]{boyd2003subgradient,Beck-2017-First} as the update rule $\{ \mathrm{update}_t \}_t$ in \cref{alg:intention-WIRL-gradual-decay} by
    $
        \mathrm{update}_t \left(\{ \theta^{t^\prime} \}_{t^\prime =1}^t | \ell , \nabla \ell \right)
        = \Proj_{\Theta} \left( \theta^t - \alpha_{t}  \nabla \ell( \theta^t) \right) .
    $
    The learning rate $\alpha_{t}$ is said to be a \emph{nonsummable step size (NSS)} if there exists a sequence $\{\beta_{t} \}_{t} \subset \mathbb{R}$ such that
        \begin{equation}
            \label{eq:seq_nonsummbale}
            \beta_{t} > 0 ,
            \quad
            \lim_{T \to \infty} \frac{\sum_{t=1}^T \beta_{t}^2}{\sum_{t=1}^T \beta_{t}} = 0
        \end{equation}
    and $\alpha_{t} = \beta_{t}$.
    The learning rate $\alpha_{t}$ is said to be a \emph{square root step size (SRSS)} if, for some $\beta > 0$, it is an NSS learning rate with $\beta_{t} = \beta t^{-1/2}$.
    The learning rate $\alpha_{t}$ is said to be a \emph{nonsummable step length (NSL)} if, for some sequence $\{\beta_{t} \}_{t} \subset \mathbb{R}$ satisfying \cref{eq:seq_nonsummbale}, we set $\alpha_{t} = \beta_{t} \| \nabla \ell(\theta^t) \|^{-1}$ when $\nabla \ell(\theta^t) \not = 0$, and $\alpha_{t} = 0$ when $\nabla \ell(\theta^t)  = 0$.
    The learning rate $\alpha_{t}$ is said to be a \emph{square root step length (SRSL)} if, for some $\beta > 0$, it is an NSL learning rate with $\beta_{t}  = \beta t^{-1/2}$.
    The learning rate $\alpha_{t}$ is said to be \emph{Polyak} if the minimum value $\min_{\Theta} \ell$ is known and
    $ 
        \alpha_{t} = 
        \left( \ell (\theta^t) - \min_{\Theta} \ell \right)\| \nabla \ell (\theta^t ) \|^{-2}
    $
    holds.
\end{example}

We present below an example satisfying \cref{assu:WIRL}, together with an implementation example of \cref{alg:intention-WIRL-gradual-decay}.
\begin{example}
    \label{exa:mip}
    In \cref{assu:WIRL}, let
    $f(x, s) = x$ (i.e., the features are the decision variables themselves).
    Under this setting, \cref{eq:FOP_linear} becomes an MILP.
    Let the sequence $\{ \mathrm{update}_t \}_t$ be PSGD, and implement the projection $\Proj_{\Delta^{d-1}}$ onto $\Theta = \Delta^{d-1}$ as in \citet{Wang-13-projection}.
    Then one can implement \cref{alg:intention-WIRL-gradual-decay}.
    This implementation is identical to \citet[Algorithm~1]{Kitaoka-2023-convergence-IRL}.
\end{example}

\begin{remark}
In \cref{assu:WIRL}, it is more natural to require that the weight space $\Theta$ does not contain $0$; see \cref{remark:mip} for the rationale.
\end{remark}

\section{Main Results}
\label{sec:main_results}

\subsection{Finite-time exact solvability of DDIOP for MILP (gradient based optimization method)}
\label{sec:solve_MILP_DDIOP_grad_theory}

For gradient-based optimization methods, we make the following assumptions: namely, that the update rule depends only on the past sequence of iterates and (sub)gradients, and that a convergence rate in terms of the best iterate is guaranteed, respectively.

\begin{assumption}
    \label{assu:gradient_based_opt_independ_ell}
    For any $t \in \mathbb{Z}_{\geq 1}$, we assume that $\mathrm{update}_{t} \colon \Theta^t \times \mathbb{R}^{t} \times(\mathbb{R}^{d})^{t} \to \Theta$ does not depend on the second component space $\mathbb{R}^{t}$.
\end{assumption}

\begin{assumption}
    \label{assu:gradient_based_opt_Q}
    For any $L$-Lipschitz-continuous convex function $\ell \colon \Theta \to \mathbb{R}$, any initial point $\theta^1 \in \Theta$, and any $t \in \mathbb{Z}_{\geq 1}$, define the iterates inductively by
    $\theta^{t+1} = \mathrm{update}_t \left( \{ \theta^{t^\prime} \}_{t^\prime =1}^t \mid \ell ,\nabla \ell \right)$.
    Then there exists a nonincreasing function $Q_{L , \theta^1 } \colon \mathbb{R}_{> 0 } \to \mathbb{R}$ such that, for every $\varepsilon > 0$, if $T \geq Q_{L , \theta^1 } (\varepsilon)$, then for any $L$-Lipschitz-continuous convex function $\ell \colon \Theta \to \mathbb{R}$, the best iterate is $\varepsilon$-accurate, that is,
    $\min_{t=1 , \ldots , T} \ell (\theta^t) - \min_{\Theta} \ell \leq \varepsilon$.
\end{assumption}

Among PSGD methods, those with NSS or NSL learning rates satisfy \cref{assu:gradient_based_opt_independ_ell}, and explicit convergence rates for the best iterate are available \citep[Propositions~9.2 and~9.8]{kitaoka2024exact}.
In contrast, for PSGD with the Polyak learning rate, $\mathrm{update}_t$ depends on the second argument, i.e., the sequence of objective values $\{\ell(\theta^{t^\prime})\}_{t^\prime=1}^t$, and hence it does not satisfy \cref{assu:gradient_based_opt_independ_ell}.

Since $X(s^{(n)})$ is a finite union of polyhedra and $f$ is such that each component function $f_i$ is piecewise affine, it follows that $f\!\left(X(s^{(n)}), s^{(n)}\right)$ is also a finite union of polyhedra \citep[Propositions~10.5 and~10.14]{kitaoka2024exact}.
Let $Y^{(n)}$ denote the set of vertices of the finite union of polyhedra $f\!\left(X(s^{(n)}), s^{(n)}\right)$ (\Cref{defi:vertex}).
\begin{definition}[{\citealp[Definition~10.15]{kitaoka2024exact}}]
    \label{defi:vertex}
    Let $X$ be a finite union of polyhedra.
    A set $Y$ is said to be the vertex set of $X$ if $Y$ is the vertex set of the convex hull $\mathrm{Conv} X$.
\end{definition}
Under \cref{assu:WIRL-uniqueness}, by the maximum principle,
\begin{equation}
    a^* (\theta^* ,s^{(n)}) \in Y^{(n)}
    \label{eq:optimal_in_vertex}
\end{equation}
holds.
We define the Lipschitz constant $L (\ell_\mathrm{sub})$ and a positive constant $\gamma (\ell_\mathrm{sub})$ as follows:
\begin{align}
    L (\ell_\mathrm{sub}) & := \sup_{\substack{ a^n \in f (X (s^{(n)}), s^{(n)}), \\ n =1, \ldots ,N }} \left\| \frac{1}{N} \sum_{n=1}^N  ( a^n - a^{(n)} ) \right\|
    \left( \leq L(f) \mathrm{diam} (\mathcal{X}) \right)
    ,
    \label{eq:Lipschitz_SL} \\
    \gamma (\ell_\mathrm{sub}) & := \max_{\theta \in \Theta} \min_{( a^{n} )_n \in \prod_{n=1}^N Y^{(n)} \setminus \{ (a^{(n)} )_n \} } \frac{1}{N} \sum_{n=1}^N \theta^\top \left( a^{(n)} - a^n \right) .
    \label{eq:gamma_SL}
\end{align}
Here, $L(f)$ denotes the Lipschitz constant of $f$, and $\mathrm{diam} (\mathcal{X}) = \sup_{x, x^\prime \in \mathcal{X}} \| x - x^\prime \|$ denotes the diameter of $\mathcal{X}$.
For an intuitive explanation of the constant $\gamma (\ell_\mathrm{sub})$, see \citet[Section~6]{kitaoka2024exact}. The positivity of $\gamma (\ell_\mathrm{sub})$ follows from the following proposition.
\begin{proposition}[{\citealp[Proposition~7.4]{kitaoka2024exact}}]
    \label{prop:grad_based_opt_achieve_min_SL_before}
    Under \Cref{assu:WIRL,assu:WIRL-uniqueness}, the constant $\gamma(\ell_{\mathrm{sub}})$ defined in \cref{eq:gamma_SL} is positive.
\end{proposition}
The theorem established by \citet{kitaoka2024exact}, on which the arguments of this paper rest, is the following.
\begin{theorem}[{\citealp[Theorem~4.3]{kitaoka2024exact}}]
    \label{theo:grad_based_opt_achieve_min_SL}
    Suppose that \cref{assu:WIRL,assu:gradient_based_opt_independ_ell,assu:gradient_based_opt_Q,assu:WIRL-uniqueness} hold.
    Then, \\
    if $T \geq Q_{L (\ell_\mathrm{sub}), \theta^1} (\gamma (\ell_\mathrm{sub}))$, we have $\min_{t=1 , \ldots , T} \ell_{\mathrm{sub}} (\theta^t ) = \min_{\Theta} \ell_{\mathrm{sub}} = 0$.
\end{theorem}

From \cref{lem:Psi_set_is_almost_Phi} and \cref{theo:grad_based_opt_achieve_min_SL}, we obtain the following corollary.

\begin{corollary}[{\citealp[Corollary~4.4]{kitaoka2024exact}}]
    \label{cor:grad_based_opt_achieve_min_SL_on_ps}
    Suppose that \cref{assu:WIRL,assu:gradient_based_opt_independ_ell,assu:gradient_based_opt_Q} hold.
    Let $\Theta = \Delta^{d-1}$.
    Then, for Lebesgue-almost every $\theta^* \in \Delta^{d-1}$ (with respect to the measure on $\Delta^{d-1}$ induced by Lebesgue measure), the following holds:
    if $T \geq Q_{L (\ell_\mathrm{sub}), \theta^1} (\gamma (\ell_\mathrm{sub}))$, then
    $\min_{t=1 , \ldots , T} \ell_{\mathrm{sub}} (\theta^t ) = \min_{\Theta} \ell_{\mathrm{sub}} = 0$.
\end{corollary}

\begin{remark}
By strengthening the result of \cref{theo:grad_based_opt_achieve_min_SL}, one can, in the case $\Theta = \Delta^{d-1}$, replace the constant $\gamma (\ell_\mathrm{sub})$ by a constant depending only on $\{ X (s^{(n)} ) \}_n$, $f$, and $\Theta$ \citep[Theorem~12.3]{kitaoka2024exact}.
\end{remark}

\subsection{Concrete complexity bounds for PSGD}
\label{sec:solve_MILP_DDIOP_grad_eg}

In this section, by applying \Cref{theo:grad_based_opt_achieve_min_SL}, we bound the number of iterations required to solve the DDIOP associated with an MILP when incorporating a concrete gradient-based optimization method into the update rule of \Cref{alg:intention-WIRL-gradual-decay}.

Let $\mathrm{diam}(\Theta)$ denote the diameter of $\Theta$.
In what follows, we use the fact that each variant of PSGD satisfies \Cref{assu:gradient_based_opt_independ_ell} and admits an error bound for the best iterate \citep[Propositions~9.2 and~9.8]{kitaoka2024exact}.
Then, by \Cref{theo:grad_based_opt_achieve_min_SL}, the following results hold.

\paragraph{PSGD (SRSS)}
When the learning rate is SRSS, one can explicitly construct the function $Q_{L,\theta^1}$ (see \citealp[Proposition~9.2]{kitaoka2024exact}); hence, one can bound the number of iterations required to solve the DDIOP associated with an MILP as follows:
\begin{corollary}[{\citealp[Corollary~4.7]{kitaoka2024exact}}]
    \label{cor:convergence_rate_NSS_eg_mini}
    Under \Cref{assu:WIRL,assu:WIRL-uniqueness}, \Cref{alg:intention-WIRL-gradual-decay} with PSGD (SRSS) solves \Cref{eq:IOP_linear} exactly within at most 
    $
        \max \left( 1, \left( \frac{ \mathrm{diam}(\Theta)^2 + (1 + \log 2) \beta^2 L(\ell_{\mathrm{sub}})^2}{\beta \gamma (\ell_{\mathrm{sub}})}  \right)^2 - 2 \right)
    $
    iterations.
\end{corollary}

\paragraph{PSGD (SRSL)}
When the learning rate is SRSL, one can explicitly construct the function $Q_{L,\theta^1}$ (see \citealp[Proposition~9.8]{kitaoka2024exact}); hence, one can bound the number of iterations required to solve the DDIOP associated with an MILP as follows:
\begin{corollary}[{\citealp[Corollary~4.9]{kitaoka2024exact}}]
    \label{cor:convergence_rate_NSL_eg_mini}
    Under \Cref{assu:WIRL,assu:WIRL-uniqueness}, \Cref{alg:intention-WIRL-gradual-decay} with PSGD (SRSL) solves \Cref{eq:IOP_linear} exactly within at most 
    $
        \max \left( 1, \left( \frac{ L(\ell_{\mathrm{sub}}) (\mathrm{diam}(\Theta)^2 + (1 + \log 2) \beta^2 )}{\beta \gamma (\ell_{\mathrm{sub}}) }  \right)^2 - 2 \right)
    $
    iterations.
\end{corollary}

\subsection{Extension: finite-time attainment of the PLF minimum}
\label{sec:finite_time_PLF_min}

As with the suboptimality loss, the number of iterations required to attain the minimum value $0$ of the PLF can also be bounded in terms of the constant $\gamma(\ell_{\mathrm{sub}})$ \citep[Theorems~4.10 and~4.11]{kitaoka2024exact}.

\begin{theorem}[{\citealp[Theorem~4.10(B)]{kitaoka2024exact}}]
    \label{theo:achieve_min_PLF_PSGD_SRSS_readable}
    Suppose that \cref{assu:WIRL,assu:WIRL-uniqueness} hold.
    Let $\{ \theta^t \}_t$ be the sequence obtained by implementing the update rule $\{ \mathrm{update}_t \}_t$ in \cref{alg:intention-WIRL-gradual-decay} via PSGD (SRSS).
    Then the condition $\ell_{\mathrm{plf} } (\theta^{T}  ) = 0$ is attained within at most
    $
        \max \left( \frac{ 4 (1 + \log 2)^2 L(\ell_{\mathrm{sub}})^4 \beta^2}{\gamma (\ell_{\mathrm{sub}})^2} -2, \left( 1 + \frac{\mathrm{diam} (\Theta)^2}{(2-\sqrt{2}) \beta \gamma (\ell_{\mathrm{sub}}) } \right)^2 \right)
    $
    iterations.
\end{theorem}

\begin{theorem}[{\citealp[Theorem~4.11(B)]{kitaoka2024exact}}]
    \label{theo:achieve_min_PLF_PSGD_SRSL_readable}
    Suppose that \cref{assu:WIRL,assu:WIRL-uniqueness} hold.
    Let $\{ \theta^t \}_t$ be the sequence obtained by implementing the update rule $\{ \mathrm{update}_t \}_t$ in \cref{alg:intention-WIRL-gradual-decay} via PSGD (SRSL).
    Then the condition $\ell_{\mathrm{plf} } (\theta^{T}  ) = 0$ holds within at most
    $
        \max \left( \frac{ 4 (1 + \log 2)^2 L(\ell_{\mathrm{sub}})^2 \beta^2}{\gamma (\ell_{\mathrm{sub}})^2} -2, \left( 1 + \frac{\mathrm{diam} (\Theta)^2 L(\ell_{\mathrm{sub}})}{(2-\sqrt{2}) \beta \gamma (\ell_{\mathrm{sub}}) } \right)^2 \right)
    $
    iterations.
\end{theorem}

Since these upper bounds are also of the form $O\!\big(\gamma(\ell_{\mathrm{sub}})^{-2}\big)$ in $\gamma(\ell_{\mathrm{sub}})$, substituting the lower bounds on $\gamma(\ell_{\mathrm{sub}})$ given in \Cref{sec:gamma_lowerbound_main} yields the number of iterations required to attain the PLF minimum as an explicit function of the problem size as well (\Cref{sec:iteration_bound_from_gamma}).

\subsection{\texorpdfstring{Lower bounds on $\gamma(\ell_{\mathrm{sub}})$ for ILPs}{Lower bounds on gamma for ILPs}}
\label{sec:gamma_lowerbound_main}

The iteration-complexity upper bounds given by \Cref{theo:grad_based_opt_achieve_min_SL}, \Cref{cor:convergence_rate_NSS_eg_mini,cor:convergence_rate_NSL_eg_mini}, and \Cref{theo:achieve_min_PLF_PSGD_SRSS_readable,theo:achieve_min_PLF_PSGD_SRSL_readable} all depend on the constant $\gamma(\ell_{\mathrm{sub}})$. Although $\gamma(\ell_{\mathrm{sub}})>0$ by \Cref{prop:grad_based_opt_achieve_min_SL_before}, estimating the number of iterations as a function of the problem size requires a lower bound on $\gamma(\ell_{\mathrm{sub}})$. In this subsection, we provide explicit lower bounds on $\gamma(\ell_{\mathrm{sub}})$ when the forward problem is an integer linear program (ILP, \Cref{assu:ILP}). Proofs and details are given in \Cref{sec:gamma_lowerbound_ILP_appx}.

In what follows, when the feasible region is given by a linear inequality system $X(s)=\{x\in\mathbb{Z}^d \mid A x\le b(s)\}$, we use the $\ell_\infty$ norm of the Graver basis $\mathcal{G}(\widetilde{A})$ (\Cref{def:graver_basis}) of the slack-augmented matrix $\widetilde{A}=[A\mid I_p]$,
$
    C:=g_\infty(\widetilde{A})=\max\{\|g\|_\infty \mid g\in\mathcal{G}(\widetilde{A})\}\in\mathbb{Z}_{\ge1},$
to describe the bounds. The main lower bounds are summarized in \Cref{tab:gamma_lowerbound} (see \Cref{sec:gamma_lowerbound_ILP_appx} for details).

\begin{table}[t]
    \caption{Lower bounds on the constant $\gamma(\ell_{\mathrm{sub}})$ for ILPs (\Cref{sec:gamma_lowerbound_ILP_appx}). Here $m=(m_i)_i$ collects the per-coordinate ranges of the features, and $C=g_\infty(\widetilde{A})$ is the $\ell_\infty$ norm of the Graver basis of $\widetilde{A}=[A\mid I_p]$. Whereas the bound for general ILPs depends on $\lVert m\rVert_2$ and $N^{d}$, the bound for linear inequality constraints is determined solely by $C$ and is independent of the right-hand side, the feature range, and the dimension-dependent exponent of the sample size.}
    \label{tab:gamma_lowerbound}
    \centering
    \renewcommand{\arraystretch}{1.5}
    \begin{tabular}{lll}
        \toprule
        Setting & $\Theta=\Delta^{d-1}$ & Reference \\
        \midrule
        General ILP & $\dfrac{1}{N^{d}\max(d-1,\sqrt2)\,\lVert m\rVert_2^{d-1}}$ & \Cref{theo:gamma_lowerbound_on_simplex} \\
        M-convex & $\Omega\!\left(\dfrac{1}{N d^{2}}\right)$ & \Cref{theo:gamma_lowerbound_on_simplex_Mconvex} \\
        M$^\natural$-convex & $\Omega\!\left(\dfrac{1}{N d^{2}}\right)$ & \Cref{theo:gamma_lowerbound_on_simplex_Mnaturalconvex} \\
        Linear inequalities (constant $C$) & $\Omega\!\left(\dfrac{1}{N\,d^{(d+1)/2}(2C)^{d-1}}\right)$ & \Cref{theo:gamma_lowerbound_linear_inequality_simplex} \\
        Totally unimodular ($C=1$) & $\Omega\!\left(\dfrac{1}{N\,d^{(d+1)/2}2^{d-1}}\right)$ & \Cref{rem:totally_unimodular_gamma} \\
        \bottomrule
    \end{tabular}
\end{table}

Here $m_i=\sup_{s \in \mathcal{S}}\big(\max_{x\in X(s)}f_i(x, s)-\min_{x\in X(s)}f_i(x, s)\big)$. For M-convex and M$^\natural$-convex sets (\Cref{assu:M-convex-projection,assu:Mnatural-convex-projection}), one obtains polynomial-in-dimension bounds that are independent of $\lVert m\rVert_2$. These bounds are obtained directly by constructing explicit test sets, without going through properties of the coefficient matrices of the describing inequality systems (such as total unimodularity), and they are far better than the bound for general linear inequality constraints (\Cref{rem:linear_inequality_includes_Mnat}).

Substituting these lower bounds into the iteration-complexity upper bounds of \Cref{cor:convergence_rate_NSS_eg_mini,cor:convergence_rate_NSL_eg_mini}, one can evaluate the number of iterations required to solve the DDIOP for MILPs~\cref{eq:IOP_linear} exactly as an explicit function of the problem size. For example, when the constraints form a totally unimodular linear inequality system, PSGD (SRSS) solves \Cref{eq:IOP_linear} exactly within at most $O\!\big(N^{2}\,d^{d+1}\,4^{d-1}\,\mathrm{poly}(\mathrm{diam}(\Theta),\beta,L(\ell_{\mathrm{sub}}))\big)$ iterations. Iteration upper bounds for the other cases are summarized in \Cref{sec:iteration_bound_from_gamma}.

\begin{remark}
    \label{rem:gamma_hardness_main}
    The lower bounds in \Cref{tab:gamma_lowerbound} can be exponentially small in the dimension $d$ (e.g., $(2C)^{-(d-1)}$). Hence they imply that the corresponding iteration-complexity upper bounds can be exponentially large, which is consistent with the computational hardness of inverse optimization; for instance, \citet[Theorem 1]{aswani2018inverse} establish NP-hardness for a formulation with noisy data (see also \Cref{sec:related_work}). The aim of this subsection is not to improve the worst-case complexity, but to quantify $\gamma(\ell_{\mathrm{sub}})$ in terms of the problem structure (the $\ell_\infty$ norm $C$ of the Graver basis of the constraint matrix).
\end{remark}

\section{Comparison tables}
\label{sec:comparison_tables}

\subsection{DDIOP for MILPs}

A performance comparison of methods for solving the DDIOP for MILPs (\cref{eq:IOP_linear}) is provided in \cref{tab:pro_con_SL_detail}.
\begin{table}[ht]
\vspace{-\intextsep} 
    \caption{Performance comparison of methods for solving the DDIOP for MILPs (\cref{eq:IOP_linear}).
    The integer $T$ denotes the number of iterations or the number of optimization calls.
    The integer $t$ denotes the iteration index or the optimization-call index.}
    \label{tab:pro_con_SL_detail}
    \centering
    \begin{tabular}{{p{3.5cm}|p{8.5cm}}}
        \toprule
        Method & Suboptimality loss \\
        \midrule
        UPA & $O\!\left( \mathrm{diam} (\Theta)\, L(\ell_{\mathrm{sub}})\, T^{-\nicefrac{1}{(d-1)} } \right)$ \citep[cf.][Theorem 2.1]{flatto1977random} \\
        \hline
        RPA & $O_{\mathbb{P}}\!\left( \mathrm{diam} (\Theta)\, L(\ell_{\mathrm{sub}})\, \left( \log T / T \right)^{\nicefrac{1}{(d-1)} }\right)$ \citep[cf.][Corollary 2.3]{reznikov2015covering} \\
        \hline
        MWU \citep{arora2012multiplicative} & $O \left( L(\ell_{\mathrm{sub}}) \log d  / \sqrt{T} \right)$, provided that $\Theta = \Delta^{d-1}$ \citep[Theorem 3.5]{Barmann-2018-online} \\
        \hline
        \citet{besbes2021online,besbes2025contextual} & $O \left( d^4 \log T /T \right)$, provided that $\Theta$ is the unit sphere and $L(\ell_{\mathrm{sub}})\leq 1$\\
        \hline
        \citet{gollapudi2021contextual} & $O \left( d \log T /T \right)$, $O \left( d^{2(d+1)} /T \right)$ provided that $\Theta$ is the unit ball and $L(\ell_{\mathrm{sub}})\leq 1$ \\
        \hline
        ONS \citep{hazan2007logarithmic} & $O \left(\mathrm{diam} (\Theta)\, L(\ell_{\mathrm{sub}})\, d\log (T/d)/T \right)$  \citep[Theorem 3.1]{sakaue2025online} \\
        \hline
        MetaGrad \citep{van2016metagrad,van2021metagrad} & $O \left( \mathrm{diam} (\Theta)\, L(\ell_{\mathrm{sub}})\, d\log (T/d)/T \right)$ \citep[Theorem 4.1]{sakaue2025online} \\
        \hline
        PSGD & $O \left( \mathrm{diam} (\Theta)\, L(\ell_{\mathrm{sub}}) / \sqrt{T} \right)$ \citep[Theorem 3.11]{Barmann-2018-online}\\
        (with step size &  \\
        $\mathrm{diam} (\Theta) L(\ell_{\mathrm{sub}})^{-1} t^{-1/2}$) &  \\
        \hline
        PSGD & $0$ if $T \geq \max \left( 1, \left( \frac{ \mathrm{diam}(\Theta)^2 + (1 + \log 2) \beta^2 L(\ell_{\mathrm{sub}})^2}{\beta \gamma (\ell_{\mathrm{sub}})}  \right)^2 -2 \right)$ \\
        (with step size $\beta t^{-1/2}$) & \citep[Corollary~4.7]{kitaoka2024exact} \\
        \hline
        PSGD & $0$ if $T \geq \max \left( 1, \left( \frac{ L(\ell_{\mathrm{sub}})\, (\mathrm{diam}(\Theta)^2 + (1 + \log 2) \beta^2 )}{\beta \gamma (\ell_{\mathrm{sub}}) }  \right)^2 -2 \right)$ \\
        (with step length $\beta t^{-1/2}$) & \citep[Corollary~4.9]{kitaoka2024exact} \\
        \bottomrule
    \end{tabular}
    \vspace{-\intextsep} 
\end{table}

\subsection{PLF minimization}

A performance comparison of methods for solving the minimum value of the PLF for MILPs is provided in \cref{tab:pro_con_PLF}.
\begin{table}[ht]
    \caption{Performance comparison of methods for minimizing the PLF for MILPs.
    The notation is the same as in \cref{tab:pro_con_SL_detail}.}
    \label{tab:pro_con_PLF}
    \centering
    \begin{tabular}{p{2.5cm}|p{10cm}}
        \toprule
        Method & Distance between the set where the PLF equals $0$ and the learned weight $\theta^T$ (when the PLF is nonzero) \\
        \midrule
        UPA & $O\!\left( \mathrm{diam} (\Theta)\, T^{-\nicefrac{1}{(d-1)} } \right)$ \\
        \hline
        RPA & $O_{\mathbb{P}}\!\left( \mathrm{diam} (\Theta)\, \left( \log T / T \right)^{\nicefrac{1}{(d-1)} }\right)$ \\
        \hline
        PSGD & $
        \mathrm{diam} (\Theta)^2
        - \beta (2-\sqrt{2}) \left( \sqrt{T} -1 \right)
        \left(
        2 \gamma (\ell_{\mathrm{sub}})
        - \beta L (\ell_{\mathrm{sub}})^2 \frac{2 ( 1 + \log 2)}{\sqrt{T+2}}
        \right)$ \\
        (with step size $\beta t^{-1/2}$) & \citep[Theorem~4.10]{kitaoka2024exact} \\
        \hline
        PSGD & $\mathrm{diam} (\Theta)^2
        - \beta (2-\sqrt{2}) \left( \sqrt{T} -1 \right)
        \left(
        2 \frac{\gamma (\ell_{\mathrm{sub}})}{L(\ell_{\mathrm{sub}})}
        - \beta \frac{2 ( 1 + \log 2)}{\sqrt{T+2}}
        \right)$ \\
        (with step length $\beta t^{-1/2}$) & \citep[Theorem~4.11]{kitaoka2024exact} \\
        \bottomrule
    \end{tabular}
\end{table}

\section{\texorpdfstring{Lower bounds on the constant $\gamma(\ell_{\mathrm{sub}})$ for ILPs}{Lower bounds on gamma for ILPs}}
\label{sec:gamma_lowerbound_ILP_appx}

In this section, we lower-bound the constant $\gamma(\ell_{\mathrm{sub}})$ defined in \cref{eq:gamma_SL}. In particular, when the forward problem has structure---such as being an integer linear program (ILP)---we give explicit lower bounds on $\gamma(\ell_{\mathrm{sub}})$, which make the iteration-complexity upper bound $Q_{L(\ell_{\mathrm{sub}}),\theta^1}(\gamma(\ell_{\mathrm{sub}}))$ of \Cref{theo:grad_based_opt_achieve_min_SL} concrete. Throughout this section, we use that \cref{eq:optimal_in_vertex} holds under \Cref{assu:WIRL,assu:WIRL-uniqueness}.

\subsection{\texorpdfstring{A min-max representation of the constant $\gamma(\ell_{\mathrm{sub}})$}{A min-max representation of gamma}}

We consider the following assumption, modeling the case where the forward problem is an integer linear program.

\begin{assumption}[Integer linear program]
    \label{assu:ILP}
    Let the decision space satisfy $\mathcal{X} \subset \mathbb{Z}^{d_{\mathcal{X}}}$, and let the feature map satisfy $f(\mathcal{X}\times\mathcal{S}) \subset \mathbb{Z}^d$.
\end{assumption}

\begin{proposition}
    \label{prop:Minimax_theorem_finite}
    Let $\Theta$ be a bounded closed convex set, and let $A \subset \mathbb{R}^d$ be a finite set. Then
    \[
        \max_{\theta \in \Theta} \min_{a \in A} \theta^{\top} a
        =
        \min_{a \in \mathrm{Conv} A} \max_{\theta \in \Theta} \theta^{\top} a
        .
    \]
\end{proposition}

\begin{proof}
    By the maximum principle,
    \[
        \max_{\theta \in \Theta} \min_{a \in A} \theta^{\top} a
        =
        \max_{\theta \in \Theta} \min_{a \in \mathrm{Conv} A} \theta^{\top} a
        .
    \]
    By the minimax theorem,
    \[
        \max_{\theta \in \Theta} \min_{a \in \mathrm{Conv} A} \theta^{\top} a
        =
        \min_{a \in \mathrm{Conv} A} \max_{\theta \in \Theta} \theta^{\top} a
        .
    \]
\end{proof}

\begin{proposition}
    \label{prop:min_max_repre_gamma}
    Suppose that \Cref{assu:WIRL,assu:WIRL-uniqueness} hold. Then
    \[
        \gamma (\ell_\mathrm{sub}) = \frac{1}{N} \min_{( a^{n} )_n \in \mathrm{Conv} \left( \prod_{n=1}^N Y^{(n)} \setminus \{ (a^{(n)} )_n \} \right) } \max_{\theta \in \Theta} \sum_{n=1}^N \theta^\top \left( a^{(n)} - a^n \right).
    \]
\end{proposition}

\begin{proof}
    This follows from \Cref{prop:Minimax_theorem_finite}.
\end{proof}

\subsubsection{\texorpdfstring{The case where $\Theta$ is the unit ball}{The case of the unit ball}}

By \Cref{prop:min_max_repre_gamma},

\begin{equation}
        \gamma (\ell_\mathrm{sub}) = \frac{1}{N} \min_{( a^{n} )_n \in \mathrm{Conv} \left( \prod_{n=1}^N Y^{(n)} \setminus \{ (a^{(n)} )_n \} \right) } \left\| \sum_{n=1}^N  a^{(n)} - \sum_{n=1}^N  a^n \right\|.
        \label{eq:gamma_ball}
\end{equation}

From this expression, the following theorem follows.
\begin{theorem}
    \label{theo:gamma_lowerbound_on_ball}
    Suppose that \Cref{assu:WIRL,assu:WIRL-uniqueness,assu:ILP} hold, and let $\Theta = B^{d}$ with $d\ge2$. 
    In addition, assume that the convex hull of the set of sums $W := \{ \sum_{n=1}^N b^n \mid (b^n)_n \in \prod_{n=1}^N Y^{(n)} \setminus \{(a^{(n)})_n\} \}$ is \emph{full-dimensional}, that is, $\dim \Conv(W) = d$. Here $W$ is the set obtained by collecting, for each choice of points $b^n$ from the vertex sets $Y^{(n)}$, the total sums $\sum_{n=1}^N b^n$, excluding the sum $P := \sum_{n=1}^N a^{(n)}$ corresponding to the observed solution. The condition $\dim \Conv(W) = d$ is equivalent to $W$ not being contained in any hyperplane of $\mathbb{R}^d$.
    Then
    \[
        \gamma (\ell_\mathrm{sub})
        \geq \frac{1}{N^d \sqrt{d-1} \| m \|_2^{d-1}}.
    \]
\end{theorem}
For the proof, see \Cref{subsec:proof_of_theo:gamma_lowerbound_on_ball}.

\begin{remark}
\label{rem:gamma_ball_fulldim_sufficient}
As a sufficient condition for the full-dimensionality condition $\dim \Conv(W) = d$, the following holds: if there exists a sample index $n_0 \in \{1, \ldots, N\}$ such that the convex hull of the vertex set with the vertex $a^{(n_0)}$ corresponding to the observed solution removed is full-dimensional, that is,
\[
    \dim \Conv\!\left( Y^{(n_0)} \setminus \{ a^{(n_0)} \} \right) = d
\]
holds, then $\dim \Conv(W) = d$ holds.
Indeed, fixing $b^n = a^{(n)}$ for $n \neq n_0$ and letting $b^{n_0} \in Y^{(n_0)} \setminus \{ a^{(n_0)} \}$ vary, the corresponding tuple differs from $(a^{(n)})_n$ since $b^{n_0} \neq a^{(n_0)}$, so its sum
\[
    \sum_{n=1}^N b^n = \Bigl( P - a^{(n_0)} \Bigr) + b^{n_0}
\]
belongs to $W$. Therefore
\[
    W \supseteq \Bigl( P - a^{(n_0)} \Bigr) + \bigl( Y^{(n_0)} \setminus \{ a^{(n_0)} \} \bigr),
\]
and since translation preserves dimension,
\[
    \dim \Conv(W)
    \ \ge\
    \dim \Conv\!\left( Y^{(n_0)} \setminus \{ a^{(n_0)} \} \right)
    = d ,
\]
that is, we obtain $\dim \Conv(W) = d$.
Intuitively, this sufficient condition requires the natural non-degeneracy, for a single sample, that ``for some single sample $s^{(n_0)}$, the image of the features $f(X(s^{(n_0)}), s^{(n_0)})$ (even after removing the observed vertex $a^{(n_0)}$) is not contained in any hyperplane of $\mathbb{R}^d$ and spreads out in all directions.''
\end{remark}

\subsubsection{\texorpdfstring{The case where $\Theta$ is the probability simplex}{The case of the probability simplex}}

By \Cref{prop:min_max_repre_gamma},

\begin{equation}
        \gamma (\ell_\mathrm{sub}) = \frac{1}{N} \min_{( a^{n} )_n \in \mathrm{Conv} \left( \prod_{n=1}^N Y^{(n)} \setminus \{ (a^{(n)} )_n \} \right) } \max_{i=1, \ldots ,d } \left( \sum_{n=1}^N  a_i^{(n)} - \sum_{n=1}^N  a_i^n \right).
        \label{eq:gamma_simplex}
\end{equation}

From this expression, the following theorem follows.
\begin{theorem}
    \label{theo:gamma_lowerbound_on_simplex}
    Suppose that \Cref{assu:WIRL,assu:WIRL-uniqueness,assu:ILP} hold, and let $\Theta = \Delta^{d-1}$ with $d\ge2$. 
    In addition, assume that $Y^{(n_0)}$ is not a singleton for some $n_0 \in \{1, \ldots, N\}$, that is, that the index set $\prod_{n=1}^N Y^{(n)} \setminus \{ (a^{(n)} )_n \}$ of the minimum in \cref{eq:gamma_SL} is nonempty (nondegeneracy).
    Then
    \[
        \gamma (\ell_\mathrm{sub})
        \geq \frac{1}{N^d \max (d-1, \sqrt{2}) \| m \|_2^{d-1}}.
    \]
\end{theorem}
For the proof, see \Cref{subsec:proof_of_theo:gamma_lowerbound_on_simplex}.

\begin{remark}
    \label{rem:nondegeneracy_simplex}
    If the nondegeneracy assumption fails, that is, if $Y^{(n)}$ is a singleton for every $n$, then the index set of the minimum in \cref{eq:gamma_SL} is empty, so under the convention that the minimum over the empty set is $+\infty$ we have $\gamma(\ell_{\mathrm{sub}})=+\infty$, and the lower bound above holds trivially (in this case $f(\cdot, s^{(n)})$ is constant on $X(s^{(n)})$ for each $n$, so $\ell_{\mathrm{sub}}\equiv0$ and every $\theta\in\Theta$ solves \cref{eq:IOP_linear}).
    We also note that, in the present theorem, the full-dimensionality assumption $\dim\Conv(W)=d$ imposed in the unit-ball case (\Cref{theo:gamma_lowerbound_on_ball}) can be replaced by nondegeneracy alone: the downward closure $Q=\Conv(W)-\mathbb{R}_{\geq0}^d$ used in the proof (\Cref{subsec:proof_of_theo:gamma_lowerbound_on_simplex}) is automatically full-dimensional as soon as $\Conv(W)\neq\emptyset$.
\end{remark}

\subsection{\texorpdfstring{$\gamma(\ell_{\mathrm{sub}})$ for discrete feasible regions}{gamma for discrete feasible regions}}

For a bounded discrete set $X \subset \mathbb{Z}^d$, a set $\mathcal{T}$ is a \emph{test set} of $X$ if the following holds: for any $x^1 \in X$, $\theta \in \Theta$ with $\theta^\top x^1 < \max_{x \in X} \theta^\top x$, there exists $g \in \mathcal{T}$ such that
\[
    x^1 + g \in X , \quad \theta^\top g > 0 .
\]

\begin{proposition}
    \label{prop:discrete_gradient_descent}
    Take a bounded discrete set $X \subset \mathbb{Z}^d$ (which is finite since it is bounded) and a test set $\mathcal{T}$ of $X$. Then, for any $x^1 \in X$ and $\theta \in \Theta$, there exist some $x^* \in \argmax_{x \in X} \theta^\top x$ and some $g^1 , \ldots , g^r \in \mathcal{T}$ such that
    \[
        x^* - x^1 = \sum_{i=1}^r g^i, \quad  \theta^\top g^i >0 .
    \]
\end{proposition}

\begin{proof}
    The set $X \subset \mathbb{Z}^d$ is bounded, hence finite. We construct a sequence $\{x^{(k)}\}_{k\ge0}$ inductively. Set $x^{(0)} := x^1$. As long as $x^{(k)} \in X$ satisfies $\theta^\top x^{(k)} < \max_{x \in X} \theta^\top x$, the definition of a test set yields some $g^{(k+1)} \in \mathcal{T}$ with
    \[
        x^{(k+1)} := x^{(k)} + g^{(k+1)} \in X, \qquad \theta^\top g^{(k+1)} > 0 .
    \]
    Then $\theta^\top x^{(k+1)} = \theta^\top x^{(k)} + \theta^\top g^{(k+1)} > \theta^\top x^{(k)}$, so $\{\theta^\top x^{(k)}\}_k$ is strictly increasing; in particular $x^{(0)}, x^{(1)}, \ldots$ are pairwise distinct points of $X$. Since $X$ is finite, this construction stops after finitely many steps $r \ge 0$. The stopping point $x^{(r)}$ does not satisfy the continuation condition $\theta^\top x^{(r)} < \max_{x \in X}\theta^\top x$, i.e., $\theta^\top x^{(r)} = \max_{x \in X}\theta^\top x$, so $x^{(r)} \in \argmax_{x \in X}\theta^\top x$. Setting $x^* := x^{(r)}$ and $g^i := g^{(i)}$ ($i = 1, \ldots, r$), telescoping gives
    \[
        x^* - x^1 = \sum_{i=1}^r \bigl(x^{(i)} - x^{(i-1)}\bigr) = \sum_{i=1}^r g^i,
        \qquad \theta^\top g^i > 0 \quad (i = 1, \ldots, r)
    \]
    (if $x^1$ is already optimal, then $r=0$ and the sum is empty).
\end{proof}

\begin{lemma}[Reduction of the lower bound via a test set]
    \label{lem:gamma_test_set_reduction}
    Suppose that \Cref{assu:WIRL,assu:WIRL-uniqueness} hold, and assume in addition that, for each $n$, $f(x,s^{(n)}) =x$, $Y^{(n)} \subseteq X(s^{(n)})$ and that $\mathcal{T}\subset\mathbb{Z}^d$ is a test set of $X(s^{(n)})$. Let $\theta^*$ be the true weight. Then, for any $\theta^\dagger\in\Theta$,
    \[
        \gamma(\ell_\mathrm{sub})
        \ \ge\ \frac{1}{N}\min_{g\in\mathcal{T},\,\theta^{*\top}g>0}\theta^{\dagger\top}g .
    \]
\end{lemma}

\begin{proof}
    If the right-hand side is nonpositive, the claim is trivial since $\gamma(\ell_\mathrm{sub})\ge0$; hence we may assume $\mu:=\min_{g\in\mathcal{T},\,\theta^{*\top}g>0}\theta^{\dagger\top}g>0$. Choosing $\theta=\theta^\dagger$ in the $\max_\theta$ of \eqref{eq:gamma_SL} gives
    \[
        \gamma(\ell_\mathrm{sub})
        \ \ge\ \frac{1}{N}\min_{(a^n)_n\ne(a^{(n)})_n}\ \sum_{n=1}^N\theta^{\dagger\top}(a^{(n)}-a^n).
    \]
    Fix any $(a^n)_n\ne(a^{(n)})_n$. By \Cref{assu:WIRL-uniqueness}, each $a^{(n)}=a^*(\theta^*,s^{(n)})$ is the unique $\theta^*$-maximizer over $Y^{(n)}$, so $\theta^{*\top}(a^{(n)}-a^n)\ge0$. Moreover $a^{(n)}=a^*(\theta^*,s^{(n)})$ is the (by \Cref{assu:WIRL-uniqueness}, unique) maximizer of $\theta^*$ over $X(s^{(n)})$, and $a^n\in X(s^{(n)})$. Applying \Cref{prop:discrete_gradient_descent} with $X=X(s^{(n)})$, $\theta=\theta^*$, and $x^1=a^n$, we obtain a decomposition toward the maximizer $a^{(n)}$, namely $a^{(n)}-a^n=\sum_i g^{n,i}$ with $g^{n,i}\in\mathcal{T}$ and $\theta^{*\top}g^{n,i}>0$. Therefore
    \[
        a^{(n)}-a^n=\sum_k\alpha^n_k g^{n,k},
        \quad \alpha^n_k\in\mathbb{Z}_{\ge0},\ g^{n,k}\in\mathcal{T},\ \theta^{*\top}g^{n,k}>0 .
    \]
    Hence $\theta^{\dagger\top}g^{n,k}\ge \mu$, so
    \[
        \sum_{n}\theta^{\dagger\top}(a^{(n)}-a^n)
        =\sum_n\sum_k\alpha^n_k\,\theta^{\dagger\top}g^{n,k}
        \ \ge\ \mu\sum_n\sum_k\alpha^n_k .
    \]
    Since the tuple differs from $(a^{(n)})_n$, there is some $n_0$ with $a^{n_0}\ne a^{(n_0)}$, hence $\sum_n\sum_k\alpha^n_k\ge1$. Therefore $\sum_n\theta^{\dagger\top}(a^{(n)}-a^n)\ge \mu$. Since $(a^n)_n$ was arbitrary, the claim follows.
\end{proof}

\subsection{\texorpdfstring{M-convex sets and linear programming}{M-convex sets and linear programming}}

\begin{definition}[M-convex set, {\citealp[cf.][]{murota2003discrete}}]
    \label{def:M-convex}
    A nonempty set $X\subseteq\mathbb{Z}^d$ is an \emph{M-convex set} if it satisfies the following exchange axiom: for all $x,y\in X$ and every index $i$ with $x_i>y_i$, there exists an index $j$ with $x_j<y_j$ such that $x-e_i+e_j\in X$ and $y+e_i-e_j\in X$ (where $e_i$ is the $i$-th unit vector). Every element of an M-convex set has the same coordinate sum $\sum_i x_i$, and M-convex sets coincide with the sets of integer points of integral base polytopes.
\end{definition}

\begin{assumption}[M-convex action set with identity features]
    \label{assu:M-convex-projection}
    Let $d_{\mathcal{X}} = d$. For each $s \in \mathcal{S}$, let $X(s) \subseteq \mathbb{Z}^d$ be an M-convex set in $\mathbb{Z}^d$, and let the feature map be the identity, i.e., $f_i(x, s) = x_i$ ($i = 1, \ldots, d$).
\end{assumption}

The indicator function of an M-convex set $X$,
$
    \delta_X (x)
    := 
        0 , \text{ if } x \in X , \,
        \infty , \text{ if } x \not \in X,
$
is an M-convex function. The following is known.

\begin{proposition}[{\citealp{murota1996convexity,murota1998discrete,murota2003discrete}}]
    \label{prop:M-convex_test_set}
    A test set of an M-convex set can be taken to be the set of single-exchange vectors $\mathcal{T} = \{ e_i - e_j \mid i \neq j \}$.
\end{proposition}

\subsubsection{\texorpdfstring{The case where $\Theta$ is the unit ball}{The case of the unit ball}}

\begin{theorem}
    \label{theo:gamma_lowerbound_on_ball_Mconvex}
    Suppose that \Cref{assu:WIRL,assu:WIRL-uniqueness,assu:M-convex-projection} hold, and let $\Theta = B^{d} = \{\theta\in\mathbb{R}^d \mid \|\theta\|_2\le 1\}$ with $d\ge 2$. Then
    \[
        \gamma (\ell_\mathrm{sub})
        \geq \frac{1}{N}\cdot\frac{2\sqrt{3}}{\sqrt{d(d^2-1)}}
        = \Omega\!\left(\frac{1}{N\,d^{3/2}}\right).
    \]
\end{theorem}

\begin{proof}
    By \Cref{prop:M-convex_test_set}, $\mathcal{T}=\{e_i-e_j\}$ is a test set. Reordering coordinates so that $\theta^*_1\ge\cdots\ge\theta^*_d$, we have $\theta^{*\top}(e_i-e_j)=\theta^*_i-\theta^*_j>0$ only for $i<j$. Set $\theta^\dagger_i:=\alpha\bigl((d+1)/2-i\bigr)$ with $\alpha:=2\sqrt3/\sqrt{d(d^2-1)}$; then $\|\theta^\dagger\|_2=1$, i.e., $\theta^\dagger\in B^d$. For each $g=e_i-e_j$ ($i<j$) with $\theta^{*\top}g>0$,
    \[
        \theta^{\dagger\top}g=\theta^\dagger_i-\theta^\dagger_j=(j-i)\alpha\ge\alpha .
    \]
    Hence, by \Cref{lem:gamma_test_set_reduction}, $\gamma(\ell_\mathrm{sub})\ge\alpha/N$.
\end{proof}

\subsubsection{\texorpdfstring{The case where $\Theta$ is the probability simplex}{The case of the probability simplex}}

\begin{theorem}
    \label{theo:gamma_lowerbound_on_simplex_Mconvex}
    Suppose that \Cref{assu:WIRL,assu:WIRL-uniqueness,assu:M-convex-projection} hold, and let $\Theta = \Delta^{d-1} = \{\theta\in\mathbb{R}_{\ge0}^d \mid \sum_{i=1}^d\theta_i = 1\}$ with $d\ge 2$. Then
    \[
        \gamma (\ell_\mathrm{sub})
        \geq \frac{1}{N}\cdot\frac{2}{d(d-1)}
        = \Omega\!\left(\frac{1}{N\,d^{2}}\right).
    \]
\end{theorem}

\begin{proof}
    By \Cref{prop:M-convex_test_set}, $\mathcal{T}=\{e_i-e_j\}$ is a test set. Reordering coordinates so that $\theta^*_1\ge\cdots\ge\theta^*_d$, we have $\theta^{*\top}(e_i-e_j)>0$ only for $i<j$. Set $\theta^\dagger_i:=(d-i)\cdot\tfrac{2}{d(d-1)}$ ($i=1,\ldots,d$); then $\theta^\dagger\ge0$ and $\sum_i\theta^\dagger_i=\tfrac{2}{d(d-1)}\sum_{i=1}^d(d-i)=1$, so $\theta^\dagger\in\Delta^{d-1}$. For each $g=e_i-e_j$ ($i<j$) with $\theta^{*\top}g>0$,
    \[
        \theta^{\dagger\top}g=\theta^\dagger_i-\theta^\dagger_j=(j-i)\cdot\tfrac{2}{d(d-1)}\ge\tfrac{2}{d(d-1)} .
    \]
    Hence, by \Cref{lem:gamma_test_set_reduction}, $\gamma(\ell_\mathrm{sub})\ge\tfrac{1}{N}\cdot\tfrac{2}{d(d-1)}$.
\end{proof}

\subsection{\texorpdfstring{M${}^\natural$-convex sets and linear programming}{M-natural-convex sets and linear programming}}

\begin{definition}[M${}^\natural$-convex set, {\citealp[cf.][]{murota2003discrete}}]
    \label{def:Mnatural-convex}
    A nonempty set $X\subseteq\mathbb{Z}^d$ is an \emph{M${}^\natural$-convex set} if it satisfies the following exchange axiom: for all $x,y\in X$ and every index $i$ with $x_i>y_i$, either (i) $x-e_i\in X$ and $y+e_i\in X$, or (ii) there exists an index $j$ with $x_j<y_j$ such that $x-e_i+e_j\in X$ and $y+e_i-e_j\in X$. M${}^\natural$-convex sets are obtained as coordinate projections of M-convex sets and coincide with the sets of integer points of generalized base polytopes.
\end{definition}

\begin{assumption}[M${}^\natural$-convex action set with identity features]
    \label{assu:Mnatural-convex-projection}
    Let $d_{\mathcal{X}} = d$. For each $s \in \mathcal{S}$, let $X(s) \subseteq \mathbb{Z}^d$ be an M${}^\natural$-convex set in $\mathbb{Z}^d$, and let the feature map be the identity, i.e., $f_i(x, s) = x_i$ ($i = 1, \ldots, d$).
\end{assumption}

The indicator function of an M${}^\natural$-convex set $X$,
$\delta_X (x)$
is an M${}^\natural$-convex function. The following is known.

\begin{proposition}[{\citealp{murota1999mconvex,murota2003discrete}}]
    \label{prop:Mnatural-convex_test_set}
    A test set of an M${}^\natural$-convex set can be taken to be the set of single-exchange vectors $\mathcal{T} = \{ e_i - e_j , \pm e_i \mid i \neq j \}$.
\end{proposition}

\subsubsection{\texorpdfstring{The case where $\Theta$ is the unit ball}{The case of the unit ball}}

\begin{theorem}
    \label{theo:gamma_lowerbound_on_ball_Mnaturalconvex}
    Suppose that \Cref{assu:WIRL,assu:WIRL-uniqueness,assu:Mnatural-convex-projection} hold, and let $\Theta = B^{d} = \{\theta\in\mathbb{R}^d \mid \|\theta\|_2\le 1\}$ with $d\ge 2$. Then
    \[
        \gamma (\ell_\mathrm{sub})
        \geq \frac{1}{N}\sqrt{\frac{6}{d(d+1)(2d+1)}}
        \geq \frac{1}{N\,d^{3/2}}
        = \Omega\!\left(\frac{1}{N\,d^{3/2}}\right).
    \]
\end{theorem}

\begin{proof}
    By \Cref{prop:Mnatural-convex_test_set}, $\mathcal{T}=\{e_i-e_j,\,\pm e_i\}$ is a test set. Reorder coordinates so that $\theta^*_1\ge\cdots\ge\theta^*_d$, and let $p,z,n$ be the numbers of positive, zero, and negative components ($p+z+n=d$, with positive components first, zero components in the middle, and negative components last). Define an integer vector $s\in\mathbb{Z}^d$ by
    \[
        s=(\,p,p-1,\ldots,1,\ \underbrace{0,\ldots,0}_{z},\ -1,\ldots,-n\,),
    \]
    and set $\theta^\dagger:=c\,s$ with $c:=(\sum_i s_i^2)^{-1/2}$ (so $\|\theta^\dagger\|_2=1$, i.e., $\theta^\dagger\in B^d$). If $\theta^*=0$, then by \Cref{remark:mip} \Cref{assu:WIRL-uniqueness} fails unless every $X(s^{(n)})$ is a singleton, in which case the index set of the minimum in \cref{eq:gamma_SL} is empty and the bound holds trivially; hence $\theta^*\neq0$, so $s\neq0$ and $c$ is well-defined. The vectors $g\in\mathcal{T}$ with $\theta^{*\top}g>0$ are of the following three types, each satisfying $\theta^{\dagger\top}g\ge c$:
    \begin{itemize}
        \item $g=e_i-e_j$ (with $\theta^*_i>\theta^*_j$): since $s_i-s_j\ge1$, $\theta^{\dagger\top}g=c(s_i-s_j)\ge c$;
        \item $g=+e_i$ (with $\theta^*_i>0$): since $s_i\ge1$, $\theta^{\dagger\top}g=cs_i\ge c$;
        \item $g=-e_i$ (with $\theta^*_i<0$): since $-s_i\ge1$, $\theta^{\dagger\top}g=c(-s_i)\ge c$.
    \end{itemize}
    Hence, by \Cref{lem:gamma_test_set_reduction}, $\gamma(\ell_\mathrm{sub})\ge c/N$. Finally, from $\sum_i s_i^2=\sum_{k=1}^p k^2+\sum_{k=1}^n k^2\le\sum_{k=1}^d k^2=\tfrac{d(d+1)(2d+1)}{6}$ (as $p+n\le d$), together with $d(d+1)(2d+1)\le6d^3$, we obtain
    \[
        c\ge\sqrt{\tfrac{6}{d(d+1)(2d+1)}}\ge d^{-3/2} .
    \]
\end{proof}

\subsubsection{\texorpdfstring{The case where $\Theta$ is the probability simplex}{The case of the probability simplex}}

\begin{theorem}
    \label{theo:gamma_lowerbound_on_simplex_Mnaturalconvex}
    Suppose that \Cref{assu:WIRL,assu:WIRL-uniqueness,assu:Mnatural-convex-projection} hold, and let $\Theta = \Delta^{d-1} = \{\theta\in\mathbb{R}_{\ge0}^d \mid \sum_{i=1}^d\theta_i=1\}$ with $d\ge 2$. Then
    \[
        \gamma (\ell_\mathrm{sub})
        \geq \frac{1}{N}\cdot\frac{2}{d(d+1)}
        = \Omega\!\left(\frac{1}{N\,d^{2}}\right).
    \]
\end{theorem}

\begin{proof}
    By \Cref{prop:Mnatural-convex_test_set}, $\mathcal{T}=\{e_i-e_j,\,\pm e_i\}$ is a test set. On $\Theta=\Delta^{d-1}$ we have $\theta^*\ge0$, so $\theta^{*\top}(-e_i)=-\theta^*_i\le0$, i.e., $g=-e_i$ does not satisfy $\theta^{*\top}g>0$. Reorder coordinates so that $\theta^*_1\ge\cdots\ge\theta^*_d$, and set
    $
        \theta^\dagger_i:=\frac{2(d-i+1)}{d(d+1)}\quad(i=1,\ldots,d).
    $
    These are positive and decreasing, with $\sum_i\theta^\dagger_i=\tfrac{2}{d(d+1)}\sum_{i=1}^d(d-i+1)=1$, so $\theta^\dagger\in\Delta^{d-1}$. The vectors $g\in\mathcal{T}$ with $\theta^{*\top}g>0$ are of the following two types, each satisfying $\theta^{\dagger\top}g\ge\tfrac{2}{d(d+1)}$:
    \begin{itemize}
        \item $g=e_i-e_j$ (with $\theta^*_i>\theta^*_j$, hence $i<j$): $\theta^{\dagger\top}g=(j-i)\tfrac{2}{d(d+1)}\ge\tfrac{2}{d(d+1)}$;
        \item $g=+e_i$ (with $\theta^*_i>0$): $\theta^{\dagger\top}g=\theta^\dagger_i\ge\theta^\dagger_d=\tfrac{2}{d(d+1)}$.
    \end{itemize}
    Hence, by \Cref{lem:gamma_test_set_reduction}, $\gamma(\ell_\mathrm{sub})\ge\tfrac{1}{N}\cdot\tfrac{2}{d(d+1)}$.
\end{proof}

\subsection{\texorpdfstring{Linear inequality constraints and the Graver basis}{Linear inequality constraints and the Graver basis}}

The individual examples so far (M-convex and M${}^\natural$-convex) share the property that $X(s)$ can be represented by an integer linear inequality system over $\mathbb{Z}^d$ and that a test set can be taken to be a Graver basis. In this subsection, we give a lower bound determined solely by the $\ell_\infty$ norm of the Graver basis of the constraint matrix $A$. In particular, the bound is \textbf{independent of the right-hand side $b(s)$, the feature range $\|m\|_2$, and the factor $N^{d-1}$ in the sample size}. This is an essential improvement over the general ILP lower bounds of \Cref{theo:gamma_lowerbound_on_ball,theo:gamma_lowerbound_on_simplex}.

\begin{assumption}[Linear inequality constraints with identity features]
    \label{assu:linear-inequality}
    Let $d_{\mathcal{X}} = d$. For a constraint matrix $A \in \mathbb{Z}^{p \times d}$ and a right-hand side $b(s) \in \mathbb{Z}^p$, let
    \[
        X(s) = \{ x \in \mathbb{Z}^d \mid A x \leq b(s) \} .
    \]
    Let the feature map be the identity, i.e., $f_i(x, s) = x_i$ ($i = 1 , \ldots , d$).
\end{assumption}

\begin{definition}[Graver basis, {\citealp[cf.][]{onn2010nonlinear,sturmfels1996grobner}}]
    \label{def:graver_basis}
    For an integer matrix $B\in\mathbb{Z}^{q\times n}$, let its integer kernel be $\ker_{\mathbb{Z}}(B):=\{z\in\mathbb{Z}^n\mid Bz=0\}$. For vectors $u,v\in\mathbb{R}^n$, we say that $u$ is \emph{conformal} to $v$, written $u\sqsubseteq v$, if $u_i v_i\ge0$ (same sign) and $|u_i|\le|v_i|$ for every coordinate $i$. The \emph{Graver basis} $\mathcal{G}(B)$ of $B$ is the (finite) set of $\sqsubseteq$-minimal elements of $\ker_{\mathbb{Z}}(B)\setminus\{\mathbf{0}\}$.
\end{definition}

Introducing slack variables converts the system into an equality system $\widetilde{A}\widetilde{x}=b(s)$ (with $\widetilde{A}=[A\mid I_p]\in\mathbb{Z}^{p\times(d+p)}$, $\widetilde{x}=(x,y)$, $y\ge0$). We define the $\ell_\infty$ norm of the Graver basis $\mathcal{G}(\widetilde{A})$ (\Cref{def:graver_basis}) by
$
    g_\infty(\widetilde{A}):=\max\{\|g\|_\infty\mid g\in\mathcal{G}(\widetilde{A})\}\in\mathbb{Z}_{\ge1}.$

This is determined by $\widetilde{A}$ (and hence by $A$) alone, and does not depend on $b(s)$, $\|m\|_2$, or $N$.

\begin{proposition}[Test set via the Graver basis]
    \label{prop:linear_inequality_test_set}
    Under \Cref{assu:WIRL,assu:linear-inequality}, let $\pi_x\colon\mathbb{Z}^{d+p}\to\mathbb{Z}^d$ be the projection onto the $x$-coordinates and set $\mathcal{T}_x:=\pi_x(\mathcal{G}(\widetilde{A}))$. Then $\mathcal{T}_x$ is a test set of $X(s)$ for every $s\in\mathcal{S}$, and $\|g\|_\infty\le g_\infty(\widetilde{A})$ for all $g\in\mathcal{T}_x$.
\end{proposition}

\begin{proof}
    Take any $s\in\mathcal{S}$, $x^1\in X(s)$, and $\theta\in\Theta$ with $\theta^\top x^1 < \max_{x\in X(s)}\theta^\top x$.
    By \Cref{assu:WIRL}, $X(s)$ is bounded, hence finite as a subset of the integer lattice, so a maximizer $x^2 \in \argmax_{x \in X(s)}\theta^\top x$ exists.
    Lifting by slacks, set $\widetilde{x}^j := (x^j,\, b(s)-Ax^j) \in \mathbb{Z}^{d+p}$ for $j=1,2$; then $\widetilde{A}\widetilde{x}^j = b(s)$ and the last $p$ components (the slack components) of $\widetilde{x}^j$ are nonnegative.
    For the difference $\widetilde{z} := \widetilde{x}^2 - \widetilde{x}^1 \in \ker_{\mathbb{Z}}(\widetilde{A})\setminus\{0\}$, the conformal decomposition property of Graver bases (every $0\neq z\in\ker_{\mathbb{Z}}(B)$ decomposes as a sum $z=\sum_k g^k$ of elements of $\mathcal{G}(B)$ with $g^k\sqsubseteq z$; \citealp[cf.][]{sturmfels1996grobner,onn2010nonlinear}) yields
    \[
        \widetilde{z} = \sum_{k=1}^r \widetilde{g}^k,
        \qquad \widetilde{g}^k \in \mathcal{G}(\widetilde{A}),
        \quad \widetilde{g}^k \sqsubseteq \widetilde{z} .
    \]
    Each $\widetilde{g}^k$ is conformal to $\widetilde{z}$, i.e., componentwise it has the same sign as $\widetilde{z}$ and its absolute value does not exceed $|\widetilde{z}|$; hence, for every subset $K \subseteq \{1,\ldots,r\}$, each component of the partial sum $\widetilde{x}^1 + \sum_{k\in K}\widetilde{g}^k$ lies between the corresponding components of $\widetilde{x}^1$ and $\widetilde{x}^2$. In particular the slack components remain nonnegative and the equality system for $\widetilde{A}$ is preserved, so the first $d$ components of the partial sum belong to $X(s)$.
    Now set $\widetilde{\theta} := (\theta, 0) \in \mathbb{R}^{d+p}$; then $\widetilde{\theta}^\top\widetilde{z} = \theta^\top(x^2 - x^1) > 0$, so there is some $k_0$ with $\widetilde{\theta}^\top \widetilde{g}^{k_0} > 0$.
    Setting $g := \pi_x(\widetilde{g}^{k_0}) \in \mathcal{T}_x$ and considering the partial sum with $K=\{k_0\}$, we obtain $x^1 + g \in X(s)$ and $\theta^\top g = \widetilde{\theta}^\top \widetilde{g}^{k_0} > 0$.
    Therefore $\mathcal{T}_x$ is a test set of $X(s)$.
    Finally, since the projection does not increase the $\ell_\infty$ norm, $\|\pi_x(\widetilde{g})\|_\infty\le\|\widetilde{g}\|_\infty\le g_\infty(\widetilde{A})$ for every $\widetilde{g}\in\mathcal{G}(\widetilde{A})$.%
    
\end{proof}

\subsubsection{\texorpdfstring{The case where $\Theta$ is the unit ball}{The case of the unit ball}}

\begin{theorem}
    \label{theo:gamma_lowerbound_linear_inequality_ball}
    Suppose that \Cref{assu:WIRL,assu:WIRL-uniqueness,assu:linear-inequality} hold, and let $\Theta=B^d=\{\theta\in\mathbb{R}^d\mid\|\theta\|_2\le1\}$ with $d\ge2$. Set $C:=g_\infty(\widetilde{A})$, and assume that the convex hull of $S^+:=\{g\in\mathcal{T}_x\mid\theta^{*\top}g>0\}$ is full-dimensional ($\dim\Conv(S^+)=d$). Then
    \[
        \gamma(\ell_\mathrm{sub})
        \geq \frac{1}{N\,\sqrt{d-1}\,(2C\sqrt{d})^{d-1}}
        = \Omega\!\left(\frac{1}{N\,d^{(d+1)/2}\,(2C)^{d-1}}\right).
    \]
\end{theorem}

\begin{proof}
    By \Cref{prop:linear_inequality_test_set}, $\mathcal{T}_x$ is a test set with $\|g\|_\infty\le C$ for all $g\in\mathcal{T}_x$. Set $S^+:=\{g\in\mathcal{T}_x\mid\theta^{*\top}g>0\}\subset\mathbb{Z}^d$.

    \emph{Step 1 (reduction to a distance via minimax).}
    \Cref{lem:gamma_test_set_reduction} gives $\gamma(\ell_\mathrm{sub})\ge\frac1N\min_{g\in S^+}\theta^{\dagger\top}g$ for any $\theta^\dagger\in\Theta$. Maximizing the right-hand side over $\theta^\dagger\in B^d$ and applying \Cref{prop:Minimax_theorem_finite} ($B^d$ is compact convex and $S^+$ is finite), we obtain
    \begin{align}        
        \gamma(\ell_\mathrm{sub})
        & \ge \frac1N\max_{\theta^\dagger\in B^d}\min_{g\in S^+}\theta^{\dagger\top}g
        = \frac1N\min_{q\in\Conv(S^+)}\max_{\|\theta^\dagger\|_2\le1}\theta^{\dagger\top}q
        \notag \\
        & = \frac1N\min_{q\in\Conv(S^+)}\|q\|_2
        = \frac1N\,\mathrm{dist}\bigl(0,\Conv(S^+)\bigr).
        \notag 
    \end{align}
    Since $\theta^{*\top}g>0$ for each $g\in S^+$, we have $\theta^{*\top}q>0$ for every $q\in\Conv(S^+)$; in particular $0\notin\Conv(S^+)$.

    \emph{Step 2 (an integral separating hyperplane).}
    $\Conv(S^+)$ is a lattice polytope whose vertices lie in $S^+\subset\{g\in\mathbb{Z}^d\mid\|g\|_\infty\le C\}$, and by assumption it is full-dimensional with $0\notin\Conv(S^+)$. Applying \Cref{prop:str_separating_hyperplane_theorem_for_polyhedra} with $A=\Conv(S^+)$ and the lattice point $P=0$ yields affinely independent vertices $y^1,\ldots,y^d\in S^+$ and an integral normal $a\in\mathbb{Z}^d$, $c\in\mathbb{Z}$ such that $\Conv(S^+)\subseteq\{x\mid a\cdot x\le c\}$ and $a\cdot0>c$ (i.e., $c\le-1$). Here $a$ is given by the cofactors of the difference vectors $v_k:=y^{k+1}-y^1$ ($k=1,\ldots,d-1$; since both endpoints satisfy $\|\cdot\|_\infty\le C$, we have $|(v_k)_j|\le2C$), and $c=a\cdot y^1$.

    \emph{Step 3 (estimate).}
    For every $q\in\Conv(S^+)\subseteq\{x\mid a\cdot x\le c\}$, the Cauchy--Schwarz inequality gives
    \[
        \|q\|_2=\|0-q\|_2\ge\frac{a\cdot(0-q)}{\|a\|}=\frac{-a\cdot q}{\|a\|}\ge\frac{-c}{\|a\|}\ge\frac{1}{\|a\|}
    \]
    (using $-c\ge1$). Hence $\mathrm{dist}(0,\Conv(S^+))\ge1/\|a\|$. Since each component of the difference vectors has width $2C$, applying \Cref{prop:norm_a_bound} with $m=(2C,\ldots,2C)$ ($\|m\|_2=2C\sqrt d$) gives $\|a\|\le\sqrt{d-1}\,(2C\sqrt d)^{d-1}$. Combining these,
    \[
        \gamma(\ell_\mathrm{sub})\ge\frac1N\cdot\frac1{\|a\|}\ge\frac{1}{N\,\sqrt{d-1}\,(2C\sqrt d)^{d-1}}.
    \]
    
\end{proof}

\subsubsection{\texorpdfstring{The case where $\Theta$ is the probability simplex}{The case of the probability simplex}}

\begin{theorem}
    \label{theo:gamma_lowerbound_linear_inequality_simplex}
    Suppose that \Cref{assu:WIRL,assu:WIRL-uniqueness,assu:linear-inequality} hold, and let $\Theta=\Delta^{d-1}$ with $d\ge2$. Set $C:=g_\infty(\widetilde{A})$. In addition, assume that $Y^{(n_0)}$ is not a singleton for some $n_0 \in \{1, \ldots, N\}$ (nondegeneracy). Then
    \[
        \gamma(\ell_\mathrm{sub})
        \geq \frac{1}{N\,\max(d-1,\sqrt2)\,(2C\sqrt d)^{d-1}}
        = \Omega\!\left(\frac{1}{N\,d^{(d+1)/2}\,(2C)^{d-1}}\right).
    \]
\end{theorem}

\begin{proof}
    By \Cref{prop:linear_inequality_test_set}, $\mathcal{T}_x$ is a test set with $\|g\|_\infty\le C$. Set $S^+:=\{g\in\mathcal{T}_x\mid\theta^{*\top}g>0\}\subset\mathbb{Z}^d$.
    By the nondegeneracy assumption we can take $a^{n} \in Y^{(n_0)} \setminus \{a^{(n_0)}\}$ (note that $Y^{(n_0)}$ is the vertex set of $\Conv(X(s^{(n_0)}))$ since $f$ is the identity map, and that $Y^{(n_0)} \subseteq X(s^{(n_0)})$). By \Cref{assu:WIRL-uniqueness}, $a^{(n_0)}$ is the unique maximizer of $\theta^*$ over $X(s^{(n_0)})$; hence, applying \Cref{prop:discrete_gradient_descent} with $X = X(s^{(n_0)})$, $\theta = \theta^*$, and $x^1 = a^{n}$, the resulting decomposition is a nonempty sum because $a^{n} \neq a^{(n_0)}$, and so there exists $g \in \mathcal{T}_x$ with $\theta^{*\top}g > 0$; that is, $S^+ \neq \emptyset$.

    \emph{Step 1 (reduction via minimax).}
    Maximizing \Cref{lem:gamma_test_set_reduction} (which holds for any $\theta^\dagger\in\Theta$) over $\theta^\dagger\in\Delta^{d-1}$ and applying \Cref{prop:Minimax_theorem_finite} ($\Delta^{d-1}$ is compact convex, $S^+$ finite), since $\max_{\theta^\dagger\in\Delta^{d-1}}\theta^{\dagger\top}q=\max_i q_i$ we obtain
    \[
        \gamma(\ell_\mathrm{sub})
        \ge\frac1N\max_{\theta^\dagger\in\Delta^{d-1}}\min_{g\in S^+}\theta^{\dagger\top}g
        =\frac1N\min_{q\in\Conv(S^+)}\max_{i=1,\ldots,d}q_i .
    \]
    From $\theta^*\in\Delta^{d-1}$ ($\theta^*\ge0$) and $\theta^{*\top}q>0$, we have $\max_i q_i>0$ for each $q\in\Conv(S^+)$, that is, $\Conv(S^+)\cap(-\mathbb{R}_{\ge0}^d)=\emptyset$.

    \emph{Step 2 (an integral separating hyperplane with a nonnegative normal).}
    Consider the upward closure $Q':=\Conv(S^+)+\mathbb{R}_{\ge0}^d$. Being the Minkowski sum of a bounded convex polytope and a convex cone, $Q'$ is a polyhedron (an intersection of finitely many closed half-spaces), and since it contains the full $d$-dimensional set $q_0+\mathbb{R}_{\ge0}^d$ for any $q_0\in\Conv(S^+)$, $Q'$ is full $d$-dimensional. Moreover $Q'$ is \emph{upward closed}, that is, $x\in Q'$ and $x'\ge x$ (componentwise) imply $x'\in Q'$. Hence each facet inequality defining $Q'$ can be taken in the form $a\cdot x\ge c$ with $a\ge0$: indeed, writing a facet inequality as $a\cdot x\ge c$ and supposing $a_i<0$ for some $i$, upward closedness gives $x+te_i\in Q'$ for all $t\ge0$ and $x\in Q'$, so $a\cdot x+t a_i\ge c$ would have to hold for all $t\ge0$, while the left-hand side tends to $-\infty$ as $t\to\infty$, a contradiction (this mirrors the argument for the downward closure in Step~3 of the proof of \Cref{theo:gamma_lowerbound_on_simplex}). As $\Conv(S^+)\cap(-\mathbb{R}_{\ge0}^d)=\emptyset$ is equivalent to $0\notin Q'$, the point $0$ violates some facet inequality: there exist an integral normal $a\in\mathbb{Z}_{\ge0}^d\setminus\{0\}$ and $c\in\mathbb{Z}$ with $\Conv(S^+)\subseteq Q'\subseteq\{x\mid a\cdot x\ge c\}$ and $a\cdot0<c$ (i.e., $c\ge1$). This facet is spanned by vertices of $S^+$ (lattice points with $\|\cdot\|_\infty\le C$) and recession directions $e_i$ (where $a_i=0$); $a$ can be constructed from the cofactors of affinely independent lattice points $y^1,\ldots,y^d$ (choosing the sign so that $a\ge0$), and the difference vectors $v_k=y^{k+1}-y^1$ satisfy $|(v_k)_j|\le 2C$.

    \emph{Step 3 (estimate).}
    For every $q\in\Conv(S^+)$, since $a\ge0$ the weights $\lambda_i:=a_i/\sum_j a_j$ form a convex combination, so
    \[
        \max_{i}q_i\ge\sum_i\lambda_i q_i=\frac{a\cdot q}{\sum_j a_j}\ge\frac{c}{\sum_j a_j}\ge\frac{1}{\sum_j a_j} .
    \]
    Applying the same Hadamard estimate as in the proof of \Cref{theo:gamma_lowerbound_on_simplex} (\Cref{prop:max_1-x_on_simplex}) with $m=(2C,\ldots,2C)$ gives $\sum_j a_j\le\max(d-1,\sqrt2)(2C\sqrt d)^{d-1}$. Therefore,
    \[
        \gamma(\ell_\mathrm{sub})\ge\frac1N\cdot\frac{1}{\sum_j a_j}\ge\frac{1}{N\,\max(d-1,\sqrt2)(2C\sqrt d)^{d-1}}.
    \]
    
\end{proof}

\begin{remark}
    \label{rem:totally_unimodular_gamma}
    If $A$ is \emph{totally unimodular} (i.e., every square submatrix of $A$ has determinant in $\{0,\pm1\}$), then $C=1$ \citep[cf.][Lemmas 3.18, 3.19]{onn2010nonlinear}. For $\Theta=\Delta^{d-1}$, by \Cref{theo:gamma_lowerbound_linear_inequality_simplex}, the lower bound on $\gamma(\ell_\mathrm{sub})$ is $\Omega\bigl(1/(N\,d^{(d+1)/2}\,2^{d-1})\bigr)$.
\end{remark}

\begin{remark}
    \label{rem:linear_inequality_includes_Mnat}
    An M${}^\natural$-convex set (\Cref{assu:Mnatural-convex-projection}) is the set of integer points of a generalized integral base polyhedron, which is described by lower and upper bound constraints over a family of subsets (submodular and supermodular inequalities). We caution that the corresponding coefficient matrix (whose rows are indicator vectors of subsets) is in general \emph{not} totally unimodular: for instance, the rows $(1,1,0)$, $(0,1,1)$, $(1,0,1)$ form a square submatrix of determinant $\pm2$ (total unimodularity holds only for special families of subsets such as laminar families). Hence the lower bound for the M${}^\natural$-convex case cannot, in general, be derived as a corollary of \Cref{rem:totally_unimodular_gamma} ($C=1$). On the other hand, \Cref{theo:gamma_lowerbound_on_ball_Mnaturalconvex,theo:gamma_lowerbound_on_simplex_Mnaturalconvex}, which exploit the explicit test set $\{e_i-e_j,\pm e_i\}$, hold regardless of the structure of the coefficient matrix of the describing inequality system, and give much sharper lower bounds, namely $\Omega(1/(N d^{3/2}))$ and $\Omega(1/(N d^{2}))$.%
    
\end{remark}

\begin{proposition}[An explicit upper bound on $g_\infty(\widetilde{A})$, {\citealp[cf.][Section 3.4]{onn2010nonlinear}}]
    \label{prop:ginfty_explicit_bound}
    For the constraint matrix $A\in\mathbb{Z}^{p\times d}$ of \Cref{assu:linear-inequality} with $A\neq0$, let $\widetilde{A}=[A\mid I_p]$. Let $\Delta(A)$ be the largest absolute value of the square subdeterminants of $A$. Then
    $
        C=g_\infty(\widetilde{A})\ \le\ d\,\Delta(A).
    $
\end{proposition}

\begin{proof}
    The matrix $\widetilde{A}=[A\mid I_p]$ has $n=d+p$ columns and $p$ rows, and contains the submatrix $I_p$, so its rank is $r=p$. Applying \citet[Lemma 3.20]{onn2010nonlinear} to $\widetilde{A}$, every $g\in\mathcal{G}(\widetilde{A})$ satisfies $\|g\|_\infty\le(n-r)\Delta(\widetilde{A})=d\,\Delta(\widetilde{A})$. Adjoining unit columns gives $\Delta(\widetilde{A})=\max(\Delta(A),1)$, and since $A\neq0$ we have $\Delta(A)\ge1$, so $\Delta(\widetilde{A})=\Delta(A)$.
\end{proof}

\subsection{\texorpdfstring{A separating hyperplane theorem}{A separating hyperplane theorem}}

\begin{definition}
    For a vector $\mathbf{b} \in \mathbb{R}^d$ ($\mathbf{b} \neq \mathbf{0}$) and a scalar $c \in \mathbb{R}$,
    $
    H = \{\mathbf{x} \in \mathbb{R}^d \mid \mathbf{b}^T \mathbf{x} \leq c\}
    $
    is called a \textbf{closed half-space}. Its \textbf{boundary hyperplane} is $\partial H = \{\mathbf{x} \in \mathbb{R}^d \mid \mathbf{b}^T \mathbf{x} = c\}$.
\end{definition}

\begin{definition}
\label{def:face_facet}
The \emph{affine hull} $\mathrm{aff}(S)$ of a set $S \subseteq \mathbb{R}^d$ is the smallest affine subspace containing $S$ (equivalently, the intersection of all affine subspaces containing $S$); for a finite list of points $y^1, \ldots, y^k \in \mathbb{R}^d$ we write $\mathrm{aff}(y^1, \ldots, y^k) := \mathrm{aff}(\{y^1, \ldots, y^k\})$.
Consider a polyhedron $A$ in $\mathbb{R}^d$ (a convex set expressed as the intersection of finitely many closed half-spaces) together with an inequality $\mathbf{b}^\top \mathbf{x} \le c$ that holds on $A$ (that is, $A \subseteq \{\mathbf{x} \mid \mathbf{b}^\top \mathbf{x} \le c\}$ with $\mathbf{b} \neq \mathbf{0}$). Then the set $F := A \cap \{\mathbf{x} \mid \mathbf{b}^\top \mathbf{x} = c\}$ is called a \emph{face} of $A$. The dimension of the affine hull $\mathrm{aff}(F)$ of a face $F$ is called the dimension of that face, and a face whose dimension equals $\dim A - 1$ is called a \emph{facet} of $A$. Moreover, when a face $F$ is a facet of $A$, the inequality $\mathbf{b}^\top \mathbf{x} \le c$ that gives it (that is, an inequality holding on $A$ such that $F = A \cap \{\mathbf{x} \mid \mathbf{b}^\top \mathbf{x} = c\}$) is called a \emph{facet-defining inequality} of $F$.
\end{definition}

\begin{proposition}
\label{prop:str_separating_hyperplane_theorem_for_polyhedra}
In Euclidean space $\mathbb{R}^d$, let $A$ be a bounded convex polytope all of whose vertices are lattice points (points of $\mathbb{Z}^d$) and that satisfies $\dim A = d$ (full-dimensional), and let $P \in \mathbb{Z}^d$ be a lattice point not belonging to $A$ ($P \notin A$). Then there exist a closed half-space $H = \{\mathbf{x} \in \mathbb{R}^d \mid \mathbf{b}^\top \mathbf{x} \leq c\}$ and affinely independent vertices $y^1, \ldots, y^{d}$ of $A$ (which, being vertices of $A$, are lattice points) such that: $A \subset H$; $P \notin H$ (i.e., $\mathbf{b}^\top P > c$); and $\partial H = \mathrm{aff}(y^1, \ldots, y^{d})$.
\end{proposition}

\begin{proof}
Since $A$ is a $d$-dimensional bounded convex polytope, it admits an irredundant facet representation
$A = \bigcap_{i=1}^{M} \{\mathbf{x} \in \mathbb{R}^d \mid \mathbf{n}_i^\top \mathbf{x} \le c_i\}$,
where each $\{\mathbf{n}_i^\top \mathbf{x} = c_i\}$ is a hyperplane defining a facet of $A$. Since $P \notin A$, there is an index $i_0$ with $\mathbf{n}_{i_0}^\top P > c_{i_0}$. The corresponding facet
\[
    G := A \cap \{\mathbf{x} \mid \mathbf{n}_{i_0}^\top \mathbf{x} = c_{i_0}\}
\]
is a $(d-1)$-dimensional face, so it has $d$ affinely independent vertices $y^1, \ldots, y^d \in G$. Being vertices of $A$, these are lattice points, and
$\mathrm{aff}(y^1, \ldots, y^d) = \mathrm{aff}(G) = \{\mathbf{n}_{i_0}^\top \mathbf{x} = c_{i_0}\}$
holds. Setting $\mathbf{b} := \mathbf{n}_{i_0}$, $c := c_{i_0}$, and $H := \{\mathbf{x} \mid \mathbf{b}^\top \mathbf{x} \le c\}$, we obtain $A \subset H$, $\mathbf{b}^\top P = \mathbf{n}_{i_0}^\top P > c_{i_0} = c$, and $\partial H = \mathrm{aff}(y^1, \ldots, y^{d})$.
\end{proof}

\subsection{\texorpdfstring{Proof of Theorem~\ref{theo:gamma_lowerbound_on_ball}}{Proof of the unit-ball lower bound}}
\label{subsec:proof_of_theo:gamma_lowerbound_on_ball}

The proof consists of three steps.

\begin{proof}[Proof of \Cref{theo:gamma_lowerbound_on_ball}]

\textbf{Step 1: Reduction to the space of sums.} 
The map $\varphi \colon (\mathbb{R}^d)^N \to \mathbb{R}^d$, $(b^1, \ldots, b^N) \mapsto \sum_{n=1}^N b^n$ is linear, so it maps convex sets to convex sets, and
\[
    \varphi\!\left(\Conv\!\left(\prod_{n=1}^N Y^{(n)} \setminus \{(a^{(n)})_n\}\right)\right) = \Conv\!\left(\varphi\!\left(\prod_{n=1}^N Y^{(n)} \setminus \{(a^{(n)})_n\}\right)\right) = \Conv(W).
\]
Hence \eqref{eq:gamma_ball} can be rewritten as
\begin{equation}
    \gamma(\ell_{\mathrm{sub}}) = \frac{1}{N} \min_{q \in \Conv(W)} \|P - q\|
    \label{eq:gamma_as_dist}
\end{equation}

Next we show $P \notin \Conv(W)$. By \Cref{assu:WIRL-uniqueness}, for each $n$, $a^{(n)} = a^*(\theta^*, s^{(n)})$ is the unique optimal solution under $\theta^*$; that is, for any $b^n \in Y^{(n)}$,
\[
    {\theta^*}^\top b^n \leq {\theta^*}^\top a^{(n)},
\]
with equality if and only if $b^n = a^{(n)}$.

If $(b^n)_n \in \prod_{n=1}^N Y^{(n)} \setminus \{(a^{(n)})_n\}$, then there is some $n_0$ with $b^{n_0} \neq a^{(n_0)}$, so
\[
    {\theta^*}^\top \sum_{n=1}^N b^n = \sum_{n=1}^N {\theta^*}^\top b^n < \sum_{n=1}^N {\theta^*}^\top a^{(n)} = {\theta^*}^\top P.
\]
This strict inequality holds for all elements of $\prod_{n=1}^N Y^{(n)} \setminus \{(a^{(n)})_n\}$, and is preserved under convex combinations. Thus ${\theta^*}^\top q < {\theta^*}^\top P$ for all $q \in \Conv(W)$, and in particular $P \notin \Conv(W)$.

\textbf{Step 2: Reduction to a separating hyperplane through lattice points.}
By \Cref{assu:ILP}, $f(\mathcal{X}\times\mathcal{S}) \subset \mathbb{Z}^d$. For each $s \in \mathcal{S}$, by translating the range of $f$, we may assume that for all $x \in X(s)$,
\[
    f(x, s) \in \prod_{i=1}^d \{0, 1, \ldots, m_i\}.
\]
Then
\begin{equation}
    P = \sum_{n=1}^N a^{(n)} \in \prod_{i=1}^d \{0, 1, \ldots, Nm_i\}, \qquad W \subset \Lambda' := \prod_{i=1}^d \{0, 1, \ldots, Nm_i\}.
    \label{eq:action_sum_in_lattice}
\end{equation}

By Step~1, $P \notin \Conv(W)$, and since $\Conv(W)$ is compact and convex, the distance between $P$ and $\Conv(W)$ is positive. As $\Conv(W)$ is the convex hull of the finite set $W$, it is a convex polytope, and $\{P\} \cap \Conv(W) = \emptyset$.

The vertices of $\Conv(W)$ are all lattice points of $W \subset \Lambda'$. By the assumption of the theorem, $\Conv(W)$ is full-dimensional ($\dim \Conv(W) = d$), and by Step~1 we have $P \notin \Conv(W)$. Hence we may apply \Cref{prop:str_separating_hyperplane_theorem_for_polyhedra} to the $d$-dimensional polytope $A = \Conv(W)$ and the lattice point $P$. This yields a closed half-space $H^- = \{x \in \mathbb{R}^d \mid a \cdot x \leq c\}$ and $d$ affinely independent points $y^1, \ldots, y^d \in W \subset \Lambda'$ that are vertices of $\Conv(W)$ (hence lattice points of $\Lambda'$) such that:
$\Conv(W) \subset H^-$,  $P \notin H^-$ (i.e., $a \cdot P > c$), and $\partial H^- = \mathrm{aff}(y^1, \ldots, y^d)$.

Since the boundary hyperplane $\partial H^- = \{x \in \mathbb{R}^d \mid a \cdot x = c\}$ is the affine hull of the lattice points $y^1, \ldots, y^d \in \mathbb{Z}^d$, the normal $a$ can be given as the cofactors of the difference vectors $v_k := y^{k+1} - y^1 \in \mathbb{Z}^d$ ($k = 1, \ldots, d-1$):
$
    a_i = (-1)^i \det(V^{(i)}) \in \mathbb{Z}$ $(i = 1, \ldots, d),
$
where $V^{(i)}$ is the matrix obtained by deleting the $i$-th column of the $(d-1) \times d$ matrix $V$ whose rows are the difference vectors. Also $c = a \cdot y^1 \in \mathbb{Z}$.

Since $P \in \mathbb{Z}^d$, we have $a \cdot P - c \in \mathbb{Z}$, which combined with $a \cdot P > c$ gives
\begin{equation}
    a \cdot P - c \geq 1.
    \label{eq:integer_gap}
\end{equation}

For any $q \in \Conv(W) \subset H^-$ we have $a \cdot q \leq c$, so by the Cauchy--Schwarz inequality,
\[
    \|P - q\| \geq \frac{a \cdot (P - q)}{\|a\|} = \frac{a \cdot P - a \cdot q}{\|a\|} \geq \frac{a \cdot P - c}{\|a\|} \geq \frac{1}{\|a\|}.
\]
Taking the minimum over $q \in \Conv(W)$,
\begin{equation}
    \min_{q \in \Conv(W)} \|P - q\| \geq \frac{1}{\|a\|}.
    \label{eq:dist_lower_by_a}
\end{equation}

\textbf{Step 3: Bounding the norm of the normal vector.}

\begin{proposition}\label{prop:max_1-x_on_simplex}
    Let $\alpha>0$ and set
    $
        F_\alpha(z):=\sum_{i=1}^d(1-z_i)^{\alpha},\text{ for } z\in\Delta^{d-1}.
    $
    Then:
    \begin{enumerate}
        \item if $\alpha\ge 1$, then $F_\alpha$ is convex on $\Delta^{d-1}$ and its maximum is attained at a vertex; in particular $F_\alpha(z)\le d-1$;
        \item if $0<\alpha<1$, then $F_\alpha$ is strictly concave on $\Delta^{d-1}$ and its maximum is attained at the barycenter $z=(1/d,\dots,1/d)$; in particular $F_\alpha(z)\le d(1-1/d)^{\alpha}$.
    \end{enumerate}
\end{proposition}

\begin{proof}
    The second derivative of each term $g(t)=(1-t)^{\alpha}$ is $g''(t)=\alpha(\alpha-1)(1-t)^{\alpha-2}$. Let $t\in[0,1)$.
    \begin{itemize}
        \item if $\alpha\ge 1$, then $g''\ge 0$, so $g$ is convex. Hence $F_\alpha$ is convex on $\Delta^{d-1}$ and its maximum is attained at a vertex $\mathbf{e}_i = (\delta_{ij})_j$, where $\delta_{ij}$ is the Kronecker delta. Then $F_\alpha(\mathbf{e}_i)=(1-1)^{\alpha}+(d-1)(1-0)^{\alpha}=d-1$, which proves the claim;
        \item if $0<\alpha<1$, then $g''<0$, so $g$ is strictly concave. Hence $F_\alpha$ is strictly concave on $\Delta^{d-1}$, and by Jensen's inequality the maximizer is the barycenter $z=(1/d,\dots,1/d)$, with $F_\alpha(1/d,\dots,1/d)=d(1-1/d)^{\alpha}$. \qedhere
    \end{itemize}
\end{proof}

\begin{proposition}
    \label{prop:leq_m-m_i_d-1}
    Let $d\ge2$ and $m=(m_1,\ldots,m_d)\in\mathbb{R}^d$. Then
    \[
    \sum_{i=1}^d\left(\sum_{j\ne i}m_j^2\right)^{d-1}
    \le (d-1) \|m\|_2^{2(d-1)}.
    \]
\end{proposition}

\begin{proof}
    If $m = 0$ the claim holds, so we may assume $m \neq 0$. The left-hand side equals
    \[
    \sum_{i=1}^d\left(\sum_{j\ne i}m_j^2\right)^{d-1}
    =
    \sum_{i=1}^d\left(\| m \|_2^2 - m_i^2\right)^{d-1}
    =
    \| m \|_2^{2(d-1)}
    \sum_{i=1}^d\left( 1 - \frac{m_i^2}{\|m \|_2^2}\right)^{d-1}
    .
    \]
    Setting $z_i = \frac{m_i^2}{\|m \|_2^2}$ and $z = (z_1, \ldots, z_d )$, we have $z \in \Delta^{d-1}$. By \Cref{prop:max_1-x_on_simplex},
    \[
        \| m \|_2^{2(d-1)}
        \sum_{i=1}^d\left( 1 - \frac{m_i^2}{\|m \|_2^2}\right)^{d-1}
        \leq
        (d-1)
        \| m \|_2^{2(d-1)},
    \]
    which proves the claim.
\end{proof}

\begin{proposition}
    \label{prop:norm_a_bound}
    Let $d\ge2$, and let $y^1,\ldots,y^d\in\mathbb{R}^d$ be affinely independent points whose difference vectors $v_k:=y^{k+1}-y^1$ ($k=1,\ldots,d-1$) satisfy $|(v_k)_j|\le m_j$ for all $j$. Let $V$ be the $(d-1)\times d$ matrix whose rows are $v_1,\ldots,v_{d-1}$, let $V^{(i)}$ be the matrix obtained by deleting the $i$-th column of $V$, and set $a_i:=(-1)^i\det(V^{(i)})$ ($i=1,\ldots,d$). Then
        \[
    \|a\|^2 \le \sum_{i=1}^d\left(\sum_{j\ne i}m_j^2\right)^{d-1}
    \leq (d-1) \| m \|_2^{2(d-1)}.
    \]
\end{proposition}

\begin{proof}
    By Hadamard's inequality,
    \[
    |a_i|=|\det(V^{(i)})|\le \prod_{k=1}^{d-1}\|v_k^{(i)}\|.
    \]
    For the norm of each $v_k^{(i)}$, since $|(v_k)_j|\le m_j$,
    \[
    \|v_k^{(i)}\|^2=\sum_{j\ne i}(v_k)_j^2\le \sum_{j\ne i}m_j^2.
    \]
    This bound is independent of $k$, so
    \[
    |a_i|\le \prod_{k=1}^{d-1}\sqrt{\sum_{j\ne i}m_j^2}
    =\left(\sum_{j\ne i}m_j^2\right)^{(d-1)/2}.
    \]
    By \Cref{prop:leq_m-m_i_d-1},
    \[
    \|a\|^2 = \sum_{i=1}^d a_i^2 \le \sum_{i=1}^d\left(\sum_{j\ne i}m_j^2\right)^{d-1}
    \leq (d-1) \| m \|_2^{2(d-1)}.
    \]
\end{proof}

With the above preparations, we complete Step~3 of the proof of \Cref{theo:gamma_lowerbound_on_ball}.
The hyperplane $\partial H^-$ passes through $d$ affinely independent lattice points $y^1, \ldots, y^d$ in $\Lambda' = \prod_{i=1}^d \{0, 1, \ldots, Nm_i\}$, and the difference vectors $v_k = y^{k+1} - y^1$ satisfy $|(v_k)_j| \leq Nm_j$.

By \Cref{prop:norm_a_bound} (applied with $m$ replaced by $Nm$),
\begin{equation}
    \|a\|^2 \leq (d-1)\,\|Nm\|_2^{2(d-1)} = (d-1)\,N^{2(d-1)}\,\|m\|_2^{2(d-1)},
    \text{ i.e., }
    \|a\| \leq \sqrt{d-1}\, N^{d-1}\, \|m\|_2^{d-1}.
    \label{eq:norm_a_bound_applied}
\end{equation}

Combining \eqref{eq:gamma_as_dist}, \eqref{eq:dist_lower_by_a}, and \eqref{eq:norm_a_bound_applied},
\begin{align*}
    \gamma(\ell_{\mathrm{sub}})
    &= \frac{1}{N} \min_{q \in \Conv(W)} \|P - q\| 
    \geq \frac{1}{N} \cdot \frac{1}{\|a\|} 
    \geq \frac{1}{N} \cdot \frac{1}{\sqrt{d-1}\, N^{d-1}\, \|m\|_2^{d-1}} 
    = \frac{1}{N^d\, \sqrt{d-1}\, \|m\|_2^{d-1}}.
\end{align*}

\end{proof}

\subsection{\texorpdfstring{Proof of Theorem~\ref{theo:gamma_lowerbound_on_simplex}}{Proof of the probability-simplex lower bound}}
\label{subsec:proof_of_theo:gamma_lowerbound_on_simplex}

The proof consists of three steps. Steps~1 and~2 are common to the proof of \Cref{theo:gamma_lowerbound_on_ball}, while Step~3 is specific to the probability-simplex case.

\begin{proof}[Proof of \Cref{theo:gamma_lowerbound_on_simplex}]

\textbf{Step 1: Reduction to the space of sums.}
By the linearity of $\varphi \colon (\mathbb{R}^d)^N \to \mathbb{R}^d$, $(b^1, \ldots, b^N) \mapsto \sum_{n=1}^N b^n$, \eqref{eq:gamma_simplex} can be rewritten as
\begin{equation}
    \gamma(\ell_{\mathrm{sub}}) = \frac{1}{N} \min_{q \in \Conv(W)} \max_{i=1,\ldots,d}(P_i - q_i).
    \label{eq:gamma_simplex_as_dist}
\end{equation}

By the same argument as in Step~1 of the proof of \Cref{theo:gamma_lowerbound_on_ball} (uniqueness via \Cref{assu:WIRL-uniqueness}), $P \notin \Conv(W)$. Moreover, by \Cref{prop:grad_based_opt_achieve_min_SL_before} ($\gamma(\ell_{\mathrm{sub}})>0$) and \cref{eq:gamma_simplex_as_dist}, for any $q \in \Conv(W)$,
\begin{equation}
    \max_{i=1,\ldots,d}(P_i - q_i) > 0.
    \label{eq:positive_max}
\end{equation}

Rewriting \eqref{eq:positive_max} in set-theoretic terms,
\begin{equation}
    \Conv(W) \cap (P + \mathbb{R}_{\geq 0}^d) = \emptyset,
    \label{eq:disjoint_orthant}
\end{equation}
where $P + \mathbb{R}_{\geq 0}^d = \{x \in \mathbb{R}^d \mid x_i \geq P_i,\ \forall\, i\}$ is the translate of the nonnegative orthant with vertex $P$.

\textbf{Step 2: Reduction to the lattice structure.}%
As in Step~2 of the proof of \Cref{theo:gamma_lowerbound_on_ball}, by \Cref{assu:ILP} and a translation,
\begin{equation}
    P \in \prod_{i=1}^d \{0, 1, \ldots, Nm_i\}, \qquad W \subset \Lambda' := \prod_{i=1}^d \{0, 1, \ldots, Nm_i\}.
    \label{eq:action_sum_in_lattice_simplex}
\end{equation}

\textbf{Step 3: Reduction to a lattice hyperplane with nonnegative normal.}
We first verify that we may assume $m_i \ge 1$ for each $i$. By \Cref{assu:ILP}, each $m_i$ is a nonnegative integer. For a coordinate $i$ with $m_i = 0$, the feature $f_i$ is constant on $X(s)$ for every $s$, so $b^n_i = a^{(n)}_i$ for every $n$ and every $b^n \in Y^{(n)}$; that is, $P_i - q_i = 0$ for every $q \in \Conv(W)$, and the coordinate $i$ does not contribute to $\max_i(P_i - q_i)$ in \cref{eq:gamma_simplex_as_dist} (by \cref{eq:positive_max}, the positive maximum is attained at a coordinate with $m_i \ge 1$). Hence, applying the argument below to the $d'$-dimensional problem obtained by removing all coordinates with $m_i = 0$ ($1 \le d' \le d$; the case $d' = 0$ is excluded, since then every $Y^{(n)}$ would be a singleton and the index set of the $\min$ in \cref{eq:gamma_simplex} would be empty), and writing $m'$ for the vector of the nonzero components of $m$, we obtain
\[
    \gamma(\ell_{\mathrm{sub}})
    \ \ge\ \frac{1}{N^{d'}\max(d'-1,\sqrt2)\,\|m'\|_2^{d'-1}}
    \ \ge\ \frac{1}{N^{d}\max(d-1,\sqrt2)\,\|m\|_2^{d-1}}
\]
(the last inequality uses $N \ge 1$, $d' \le d$, and $\|m'\|_2 = \|m\|_2 \ge 1$). In what follows we assume $m_i \ge 1$ for each $i$. The estimate below (\Cref{prop:leq_m-m_i_d-1,prop:norm_a_bound}) requires the reduced dimension to satisfy $d'\ge2$; if $d'=1$, then $\Delta^{0}$ is a single point and, since $W\subset\mathbb{Z}$ and $P\in\mathbb{Z}$ with $P-q>0$ for all $q\in\Conv(W)$ by \cref{eq:positive_max}, integrality gives $\min_{q\in\Conv(W)}(P-q)\ge1$, whence $\gamma(\ell_{\mathrm{sub}})\ge 1/N\ge 1/(N\sqrt2)$ and the bound already holds. We may therefore assume $d'\ge2$. Consider the downward closure
\[
    Q := \Conv(W) - \mathbb{R}_{\geq 0}^d = \{ q - v \mid q \in \Conv(W),\ v \in \mathbb{R}_{\geq 0}^d \}.
\]
Since $\Conv(W)$ is a bounded convex polytope and $\mathbb{R}_{\geq 0}^d$ is a convex cone, $Q$ is a polyhedron (an intersection of finitely many closed half-spaces); and since for any $q_0 \in \Conv(W)$ it contains $Q \supseteq q_0 - \mathbb{R}_{\geq 0}^d$ (which is full $d$-dimensional), $Q$ is also full $d$-dimensional. Moreover $Q$ is \emph{downward closed}, that is, $x \in Q$ and $x' \le x$ (componentwise) imply $x' \in Q$. Hence the normal $a$ of each facet inequality $a \cdot x \le c$ defining $Q$ can be taken to satisfy $a \ge 0$ (if $a_i < 0$, then for $x \in Q$ downward closedness gives $x - t e_i \in Q$ for all $t \ge 0$, so $a \cdot x - t a_i \le c$ holds for all $t \ge 0$, a contradiction as $t \to \infty$).
\eqref{eq:disjoint_orthant} ($\Conv(W) \cap (P + \mathbb{R}_{\geq 0}^d) = \emptyset$) is equivalent to $P \notin Q$, and since the full $d$-dimensional $Q$ equals the intersection of the finitely many closed half-spaces defining its facets (\Cref{def:face_facet}), $P$ violates some facet inequality. That is, there exist $a^\circ \in \mathbb{R}_{\geq 0}^d \setminus \{0\}$ and $c^\circ \in \mathbb{R}$ such that
\[
    \Conv(W) \subseteq Q \subseteq \{ x \mid a^\circ \cdot x \leq c^\circ \},
    \qquad a^\circ \cdot P > c^\circ.
\]
The facet $F := Q \cap \{ x \mid a^\circ \cdot x = c^\circ \}$ is $(d-1)$-dimensional (the vertices of $Q$ coincide with vertices of $\Conv(W)$, hence are lattice points of $\Lambda'$), and its edge directions are either (i) difference vectors between vertices of $\Conv(W)$, or (ii) coordinate directions $-e_i$ (for $i$ with $a^\circ_i = 0$; by downward closedness $F$ may extend as a ray in this direction). Since $\mathrm{aff}(F) = \{ x \mid a^\circ \cdot x = c^\circ \}$ is $(d-1)$-dimensional, a vertex $y^1 \in W$ of $F$ together with $d-1$ affinely independent vectors taken from these edge directions yields affinely independent lattice points $y^1, \ldots, y^d \in \mathbb{Z}^d$ with $\mathrm{aff}(y^1, \ldots, y^d) = \{ x \mid a^\circ \cdot x = c^\circ \}$. The difference vectors $v_k := y^{k+1} - y^1$ satisfy, in case (i), $|(v_k)_j| \leq Nm_j$ (both endpoints are lattice points of $\Lambda'$), and in case (ii), $v_k = -e_i$, hence $|(v_k)_j| = \delta_{ij} \leq 1 \leq Nm_j$.
Moreover, since $a^\circ \geq 0$ and $a^\circ \cdot P > c^\circ$, for any $x \in \{a^\circ \cdot x = c^\circ\}$ we have $\sum_i a^\circ_i (P_i - x_i) = a^\circ \cdot P - c^\circ > 0$, so there is some $j$ with $a^\circ_j > 0$ and $P_j - x_j > 0$, whence $\max_i(P_i - x_i) > 0$.

Therefore, we obtain a hyperplane $H := \mathrm{aff}(y^1,\ldots,y^d)$ ($= \{x \mid a^\circ \cdot x = c^\circ\}$) containing $d$ affinely independent lattice points $y^1, \ldots, y^d$ with difference vectors in $\Lambda' = \prod_{i=1}^d \{0, 1, \ldots, Nm_i\}$ such that:
\begin{itemize}
    \item $\Conv(W)$ is contained in the closed half-space $H^- = \{x \mid a \cdot x \leq c\}$ defined by $H$ (with $a \cdot P > c$);
    \item $\forall\, x \in H,\ \max_{i=1,\ldots,d}(P_i - x_i) > 0$.
\end{itemize}

Construct the normal vector $a$ as the cofactors of the $(d-1) \times d$ integer matrix $V$ whose rows are the difference vectors $v_k := y^{k+1} - y^1 \in \mathbb{Z}^d$ ($k = 1, \ldots, d-1$):
$
    a_i = (-1)^i \det(V^{(i)}) \in \mathbb{Z} \qquad (i = 1, \ldots, d).
$
Then $c = a \cdot y^1 \in \mathbb{Z}$. This $a$ is an integer normal of the hyperplane $\mathrm{aff}(y^1,\ldots,y^d) = \{x \mid a^\circ \cdot x = c^\circ\}$, and is parallel to $a^\circ$. Since $H \cap (P + \mathbb{R}_{\geq 0}^d) = \emptyset$, choosing the sign of $a$ appropriately yields
\begin{equation}
    a_i \geq 0 \quad (\forall\, i), \qquad a \cdot P - c \geq 1.
    \label{eq:sign_condition_simplex}
\end{equation}

For any $q \in \Conv(W) \subset H^-$, setting $\lambda_i := \frac{a_i}{\sum_{j=1}^d a_j} \geq 0$ (so $\sum_{i=1}^d \lambda_i = 1$),
\begin{equation}
    \max_{i=1,\ldots,d}(P_i - q_i) \geq \sum_{i=1}^d \lambda_i (P_i - q_i) = \frac{a \cdot (P - q)}{\sum_{j=1}^d a_j} \geq \frac{a \cdot P - c}{\sum_{j=1}^d a_j} \geq \frac{1}{\sum_{j=1}^d a_j}.
    \label{eq:convex_lower_simplex}
\end{equation}

We bound $\sum_{j=1}^d a_j$. The difference vectors satisfy $|(v_k)_j| \leq Nm_j$, so by Hadamard's inequality,
\[
    a_i \leq |a_i| \leq \left(\sum_{j \neq i} (Nm_j)^2\right)^{(d-1)/2}.
\]
Since $a_i \geq 0$,
\[
    \sum_{i=1}^d a_i \leq \sum_{i=1}^d \left(\sum_{j \neq i} (Nm_j)^2\right)^{(d-1)/2} = N^{d-1} \|m\|_2^{d-1} \sum_{i=1}^d \left(1 - \frac{m_i^2}{\|m\|_2^2}\right)^{(d-1)/2}.
\]
By \Cref{prop:max_1-x_on_simplex},
\begin{equation}
    \sum_{i=1}^d a_i \leq \max(d-1, \sqrt{2})\, N^{d-1}\, \|m\|_2^{d-1}.
    \label{eq:sum_a_bound_simplex}
\end{equation}

Combining \eqref{eq:gamma_simplex_as_dist}, \eqref{eq:convex_lower_simplex}, and \eqref{eq:sum_a_bound_simplex},
\begin{align*}
    \gamma(\ell_{\mathrm{sub}})
    &= \frac{1}{N} \min_{q \in \Conv(W)} \max_{i=1,\ldots,d}(P_i - q_i) 
    \geq \frac{1}{N} \cdot \frac{1}{\sum_{j=1}^d a_j} \\
    &\geq \frac{1}{N} \cdot \frac{1}{\max(d-1, \sqrt{2})\, N^{d-1}\, \|m\|_2^{d-1}} 
    = \frac{1}{N^d\, \max(d-1, \sqrt{2})\, \|m\|_2^{d-1}}.
\end{align*}

\end{proof}

\subsection{\texorpdfstring{Tightness of the general-ILP lower bounds}{Tightness of the general-ILP lower bounds}}
\label{subsec:tightness}

The lower bounds of \Cref{theo:gamma_lowerbound_on_ball,theo:gamma_lowerbound_on_simplex} decay exponentially in the dimension $d$. The following explicit family shows that $\gamma(\ell_{\mathrm{sub}})$ itself can be exponentially small in $d$, so that the iteration upper bound $T=O(1/\gamma(\ell_{\mathrm{sub}})^2)$ can be exponentially large in $d$.

\begin{proposition}
\label{prop:tightness}
Fix an integer $K\ge 2$ and a dimension $d\ge 2$, and set $a=(1,K,K^2,\ldots,K^{d-1})\in\mathbb{Z}^d$. Consider the single-sample DDIOP ($N=1$) with identity features $f(x,s)=x$ and feasible region
\begin{align*}
    & X =\{p_0,p_1,\ldots,p_{d-1}, p_d,\,a^*\}\subset\mathbb{Z}^d,
    \\
    & p_0=0,\quad p_i=K e_i-e_{i+1}\ (i=1,\ldots,d-1),\quad p_d=-e_d,\quad a^*=e_1,
\end{align*}
whose feature ranges satisfy $m_i\le K+1$ for every $i$, and let the true weight be $\theta^*=a/\|a\|$ (normalized in $\|\cdot\|_1$ for $\Theta=\Delta^{d-1}$ and in $\|\cdot\|_2$ for $\Theta=B^{d}$). Then $a^*$ is the unique maximizer of $\theta^{*\top}x$ over $X$, and, with $\|a\|_1=\tfrac{K^{d}-1}{K-1}$ and $\|a\|_2=\sqrt{\tfrac{K^{2d}-1}{K^2-1}}$,
\[
    \tfrac{1}{\|a\|_1}\le \gamma(\ell_{\mathrm{sub}})\le \tfrac{1}{K^{d-1}}\quad(\Theta=\Delta^{d-1}),
    \qquad
    \tfrac{1}{\|a\|_2}\le \gamma(\ell_{\mathrm{sub}})\le \tfrac{1}{K^{d-1}}\quad(\Theta=B^{d}).
\]
In particular $\gamma(\ell_{\mathrm{sub}})=\Theta(K^{-(d-1)})$ in both cases.
\end{proposition}

\begin{proof}
Since $a\cdot p_0=0$, $a\cdot p_i=K\cdot K^{i-1}-K^{i}=0$ ($i=1,\ldots,d-1$), $a\cdot p_d=-K^{d-1}$, and $a\cdot a^*=1$, we have $\theta^{*\top}a^*=1/\|a\|>0\ge\theta^{*\top}p$ for every $p\in\{p_0,\ldots,p_{d-1},p_d\}$, so $a^*$ is the unique maximizer; thus $a^{(1)}=a^*$ and \Cref{assu:WIRL-uniqueness} holds. Put $W=X\setminus\{a^*\}=\{p_0,\ldots,p_{d-1},p_d\}$. Since $p_0,\ldots,p_{d-1}$ span the hyperplane $\{a\cdot x=0\}$ and $p_d=-e_d\notin\{a\cdot x=0\}$, $\Conv(W)$ is full-dimensional, so \Cref{theo:gamma_lowerbound_on_ball} applies to this family. By \Cref{prop:min_max_repre_gamma}, $\gamma(\ell_{\mathrm{sub}})=\min_{q\in\Conv(W)}\max_{\theta\in\Theta}\theta^\top(a^*-q)$. Every $q\in\Conv(W)$ satisfies $a\cdot q\le0$, whereas $a\cdot a^*=1$.

\emph{Lower bounds.} If $\Theta=\Delta^{d-1}$, then $\max_{\theta}\theta^\top(a^*-q)=\max_i(a^*_i-q_i)$, and from $\sum_i a_i(a^*_i-q_i)=a\cdot(a^*-q)=1-a\cdot q\ge1$ with $a_i>0$ we get $\max_i(a^*_i-q_i)\ge 1/\sum_i a_i=1/\|a\|_1$. If $\Theta=B^{d}$, then $\max_\theta\theta^\top(a^*-q)=\|a^*-q\|_2\ge|a\cdot(a^*-q)|/\|a\|_2\ge1/\|a\|_2$.

\emph{Upper bound.} Take $q^\circ=\lambda_0 p_0+\sum_{i=1}^{d-1}K^{-i}p_i$ with $\lambda_0=1-\sum_{i=1}^{d-1}K^{-i}\ge0$. A direct computation gives $q^\circ=e_1-K^{-(d-1)}e_d\in\Conv(W)$, so $a^*-q^\circ=K^{-(d-1)}e_d$, whence $\max_i(a^*_i-q^\circ_i)=K^{-(d-1)}$ and $\|a^*-q^\circ\|_2=K^{-(d-1)}$; thus $\gamma(\ell_{\mathrm{sub}})\le K^{-(d-1)}$ in both cases. Finally $\|a\|_1=\tfrac{K^{d}-1}{K-1}\le 2K^{d-1}$ and $\|a\|_2=\sqrt{\tfrac{K^{2d}-1}{K^2-1}}\le\tfrac{K}{\sqrt{K^2-1}}\,K^{d-1}\le\tfrac{2}{\sqrt3}\,K^{d-1}$ (both using $K\ge2$), so $\gamma(\ell_{\mathrm{sub}})=\Theta(K^{-(d-1)})$ with constants independent of $d$.
\end{proof}

\section{\texorpdfstring{Iteration upper bounds for integer linear programs}{Iteration upper bounds for integer linear programs}}
\label{sec:iteration_bound_from_gamma}

The iteration upper bound of PSGD (SRSS) (\Cref{cor:convergence_rate_NSS_eg_mini}) and that of PSGD (SRSL) (\Cref{cor:convergence_rate_NSL_eg_mini}) are both of the form $O\!\big(\gamma(\ell_{\mathrm{sub}})^{-2}\big)\times\mathrm{poly}(\mathrm{diam}(\Theta),\beta,L(\ell_{\mathrm{sub}}))$. Hence, substituting each lower bound on $\gamma(\ell_{\mathrm{sub}})$ from \Cref{tab:gamma_lowerbound} (obtained in \Cref{sec:gamma_lowerbound_ILP_appx}), the SRSS and SRSL iteration upper bounds for solving \Cref{eq:IOP_linear} exactly are obtained as explicit functions of the problem size (both share the same $O$ order). We summarize the results in \Cref{tab:iteration_upperbound} (for the probability simplex $\Theta=\Delta^{d-1}$; we abbreviate $\mathrm{poly}=\mathrm{poly}(\mathrm{diam}(\Theta),\beta,L(\ell_{\mathrm{sub}}))$). The totally unimodular case ($O(N^{2}\,d^{d+1}4^{d-1}\,\mathrm{poly})$) illustrated in \Cref{sec:gamma_lowerbound_main} corresponds to the last row of this table.

\begin{table}[t]
    \caption{Iteration upper bounds of PSGD (SRSS) and PSGD (SRSL), obtained by substituting each lower bound on $\gamma(\ell_{\mathrm{sub}})$ from \Cref{tab:gamma_lowerbound} into \Cref{cor:convergence_rate_NSS_eg_mini,cor:convergence_rate_NSL_eg_mini} ($\Theta=\Delta^{d-1}$, $\mathrm{poly}=\mathrm{poly}(\mathrm{diam}(\Theta),\beta,L(\ell_{\mathrm{sub}}))$; both share the same $O$ order).}
    \label{tab:iteration_upperbound}
    \centering
    \renewcommand{\arraystretch}{1.5}
    
    \begin{tabular}{ll}
        \toprule
        Setting & Iteration upper bound (SRSS, SRSL) \\
        \midrule
        ILP & $O\!\big(N^{2d}\,d^{2}\,\lVert m\rVert_2^{2(d-1)}\,\mathrm{poly}\big)$ \\
        M-convex & $O\!\big(N^{2}d^{4}\,\mathrm{poly}\big)$ \\
        M$^\natural$-convex & $O\!\big(N^{2}d^{4}\,\mathrm{poly}\big)$ \\
        Linear inequality ($C=g_\infty(\widetilde{A})$) & $O\!\big(N^{2}\,d^{d+1}(2C)^{2(d-1)}\,\mathrm{poly}\big)$ \\
        Totally unimodular ($C=1$) & $O\!\big(N^{2}\,d^{d+1}4^{d-1}\,\mathrm{poly}\big)$ \\
        \bottomrule
    \end{tabular}
    
\end{table}

For M-convex and M$^\natural$-convex sets the iteration count is bounded by a polynomial in the dimension $d$, whereas for general ILPs and linear inequality constraints it can grow exponentially in $d$ (\Cref{rem:gamma_hardness_main}).

The iteration upper bounds for attaining the PLF minimum (\Cref{theo:achieve_min_PLF_PSGD_SRSS_readable,theo:achieve_min_PLF_PSGD_SRSL_readable}) are also of the form $O\!\big(\gamma(\ell_{\mathrm{sub}})^{-2}\big)\times\mathrm{poly}(\mathrm{diam}(\Theta),\beta,L(\ell_{\mathrm{sub}}))$. Hence, substituting each lower bound on $\gamma(\ell_{\mathrm{sub}})$ from \Cref{tab:gamma_lowerbound} yields the iteration upper bound for attaining the PLF minimum as an explicit function of the problem size as well, whose $O$ order coincides with that in \Cref{tab:iteration_upperbound}.

\section{Conclusion}
\label{sec:conclusion}

In this paper, for data-driven inverse optimization whose forward problem is an integer linear program (ILP), we gave explicit lower bounds on the geometric constant $\gamma(\ell_{\mathrm{sub}})$ of the suboptimality loss in terms of the structure of integer programming (total unimodularity, Graver bases, and M-convexity/M$^\natural$-convexity). Substituting them into the iteration upper bounds for gradient-based optimization methods~\citep{kitaoka2024exact}, we evaluated the number of iterations sufficient to achieve exact consistency with the observed data as an explicit function of the sample size, the dimension, the feature ranges, and the structure of the constraint matrix, up to polynomial factors in basic constants. Similarly, we gave an explicit bound on the number of iterations required to attain the minimum of the prediction loss of features (PLF). The obtained lower bounds show that the number of iterations is bounded by a polynomial in the problem size when structure can be exploited, while for general ILPs they can become exponentially small in the dimension; the latter is consistent with the NP-hardness of inverse optimization with noisy data established by \citet{aswani2018inverse}. In fact \Cref{prop:tightness} exhibits an explicit family of ILP instances on which $\gamma(\ell_{\mathrm{sub}})$ is exponentially small in $d$; hence no lower bound on $\gamma(\ell_{\mathrm{sub}})$ polynomial in the problem size can hold for general ILPs, and the exponential dependence of the iteration bound on $d$ is genuine. Sharpening the general ILP lower bound and improving its $N^{d}$ dependence are left for future work.%

\section*{Acknowledgement }
    We also thank GPT-5.2, GPT-5.4, Opus 4.7, and Opus 4.8, Fable 5 for their assistance with proofreading the manuscript.

\bibliographystyle{apalike} 
\bibliography{suboptimality_loss.bib} 


\newpage

\appendix
\crefalias{section}{appendix}

\section{Regret analysis for the suboptimality loss}
\label{sec:known_regret_analysis}

Let $\Theta$ be a nonempty subset of $\mathbb{R}^d$.
Let $\ell\colon\Theta\to\mathbb{R}$ be a loss function.
Consider a sequence $\{\theta^t\}_{t=1}^T\subset\Theta$.
We define the regret by
\[
    \mathrm{Regret}(T)
    = \sum_{t=1}^T \left(\ell(\theta^{t}) - \min_{\theta\in\Theta}\ell(\theta)\right) \, .
\]
In what follows, we review existing work on regret analyses for the suboptimality loss.
The results are summarized in \cref{tab:pro_con_SL_regret_detail}.

\begin{table}[ht]
\vspace{-\intextsep} 
    \caption{Performance comparison of methods for solving the DDIOP for MILPs (\cref{eq:IOP_linear}).
    The integer $T$ denotes the number of iterations or the number of optimization calls.
    The integer $t$ denotes the iteration index or the optimization-call index.}
    \label{tab:pro_con_SL_regret_detail}
    \centering
    \begin{tabular}{p{3cm}|p{9cm}}
        \toprule
        Method & Regret for the suboptimality loss \\
        \midrule
        MWU \citep{arora2012multiplicative} & $O \left( L(\ell_{\mathrm{sub}}) \log d  \sqrt{T} \right)$, provided that $\Theta = \Delta^{d-1}$ \citep[Theorem 3.5]{Barmann-2018-online} \\
        \hline
        PSGD (online gradient descent) & $O \left( \mathrm{diam} (\Theta) L(\ell_{\mathrm{sub}}) \sqrt{T} \right)$ \citep[Theorem 3.11]{Barmann-2018-online} \\
        (with step size &  \\
        $\mathrm{diam} (\Theta) L(\ell_{\mathrm{sub}})^{-1} t^{-1/2}$) &  \\
        \hline
        \citet{besbes2021online,besbes2025contextual} & $O \left( d^4 \log T  \right)$, provided that $\Theta$ is the unit sphere and $L(\ell_{\mathrm{sub}})\leq 1$\\
        \hline
        \citet{gollapudi2021contextual} & $O \left( d \log T \right)$, provided that $\Theta$ is the unit ball and $L(\ell_{\mathrm{sub}})\leq 1$ \\
         & $O \left( d^{2(d+1)} \right)$, provided that $\Theta$ is the unit ball and $L(\ell_{\mathrm{sub}})\leq 1$ \\
        \hline
        ONS \citep{hazan2007logarithmic} & $O \left(\mathrm{diam} (\Theta) L(\ell_{\mathrm{sub}}) d\log (T/d) \right)$  \citep[Theorem 3.1]{sakaue2025online}\\
        \hline
        MetaGrad \citep{van2016metagrad,van2021metagrad} & $O \left( \mathrm{diam} (\Theta) L(\ell_{\mathrm{sub}}) d\log (T/d) \right)$ \citep[Theorem 4.1]{sakaue2025online} \\
        \bottomrule
        Cf.\ lower bound & $\Omega (d)$ \citep[\S 5]{sakaue2025online}\\
    \end{tabular}
    \vspace{-\intextsep} 
\end{table}

\section{Online and offline optimization}

Let $\Theta$ be a nonempty subset of $\mathbb{R}^d$.
Let $\ell\colon\Theta\to\mathbb{R}$ be a function.
Consider a sequence $\{\theta^t\}_{t=1}^T\subset\Theta$.
Then, the following holds:
\begin{equation}
    \min_{t=1,\ldots,T}
    \left(\ell(\theta^{t}) - \min_{\Theta}\ell\right)
    \leq
    \frac{1}{T}
    \sum_{t=1}^T \left(
        \ell(\theta^t) - \min_{\Theta}\ell
    \right)
    \leq \frac{\mathrm{Regret}(T)}{T}
    \, .
    \label{eq:offline-online-estimate}
\end{equation}

By applying the existing regret bounds reviewed in \cref{sec:known_regret_analysis} to \cref{eq:offline-online-estimate}, one can upper-bound the best-iterate performance for the suboptimality loss.
For the resulting bounds, see \cref{tab:pro_con_SL_detail}.

\section{Additional remarks on Assumption \ref{assu:WIRL}}

\begin{remark}\label{remark:mip}
We explain why it is preferable to exclude $0$ from the weight space $\Theta$.
At the origin $0$, one has $\ell_{\mathrm{sub}}(0)=0$, and thus $0$ attains the minimum value of the suboptimality loss.
Consequently, for the DDIOP for MILPs (\cref{eq:IOP_linear}), $\theta=0$ is a solution.

However, the optimizer $x^*(0,s^{(n)})$ may be any point in $X(s^{(n)})$; hence, unless $X(s^{(n)})$ is a singleton, the optimizer is not uniquely determined.
Therefore, even if the true weight were $\theta^*=0$, Assumption~\ref{assu:WIRL-uniqueness} would not be satisfied.
That is, even if learning returns $\theta=0$, such a $\theta$ does not coincide with the true weight $\theta^*$.
This issue is common to many IOPs for LPs, regardless of whether the setting is online or offline \citep{bertsimas2015data,Mohajerin-2018-Data,Barmann-2018-online,Chen-2020-Online,sun2023maximum,Kitaoka-2023-convergence-IRL,Kitaoka-2023-imitation-WIRL,kitaoka2024exact}.
\end{remark}

\end{document}